\documentclass[12pt]{amsart}
\usepackage{amssymb, latexsym,framed}
\usepackage{csquotes}
\usepackage{amsthm}
\usepackage{amsmath,amscd}
\usepackage{esint}
\usepackage{amsfonts}
\usepackage[title]{appendix}
\usepackage[dvips]{graphicx}
\usepackage{hyperref}
\usepackage{enumitem}
\usepackage{dsfont}
\usepackage[dvips]{graphicx}
\usepackage{color}
\numberwithin{equation}{section} 
\newcommand{\bean}{\begin{eqnarray*}}
\newcommand{\eean}{\end{eqnarray*}}
\newcommand{\be}{\begin{equation}}

\newcommand{\ee}{\end{equation}}
\newcommand{\bd}{\begin{displaymath}}
\newcommand{\ed}{\end{displaymath}}

\newcommand{\Om}{\Omega}

\newcommand{\beq}{\begin{equation}}
\newcommand{\eeq}{\end{equation}}
\newcommand{\bea}{\begin{eqnarray}}
\newcommand{\eea}{\end{eqnarray}}

\newcommand{\abs}[1]{\left\vert{#1}\right\vert}

\newcommand{\R}{\mathbb{R}}

\newcommand{\e}{\varepsilon}
\newcommand{\norm}[1]{\left\Vert#1\right\Vert}

\DeclareMathOperator{\dist}{dist}

\DeclareMathOperator{\tr}{tr}

\DeclareMathOperator*{\osc}{osc}

\newtheorem{thm}{Theorem}[section]
\newtheorem{prop}[thm]{Proposition}
\newtheorem{lem}[thm]{Lemma}
\newtheorem*{lem*}{Lemma}
\theoremstyle{definition}
\newtheorem{rem}[thm]{Remark}

\def\({\left(}
\def\){\right)}
\def\1{\mathds{1}}

\def\l|{\left|}

\def\r|{\right|}

\def\XXint#1#2#3{{\setbox0=\hbox{$#1{#2#3}{\int}$ }
\vcenter{\hbox{$#2#3$ }}\kern-.6\wd0}}

\begin{document}

\begin{abstract} We establish { small energy} H\"{o}lder bounds for  minimizers $u_\e$ of 
\[E_\e (u):=\int_\Om  W(\nabla u)+ \frac{1}{\e^2} \int_\Om f(u),\]
 where $W$ is a positive definite quadratic form { and the potential $f$ constrains $u$ to be close to a given manifold $\mathcal N$}. This implies that, up to subsequence, $u_\e$ converges locally uniformly to { an $\mathcal N$-valued} $W$-harmonic map, away from its singular set.
We treat general energies, covering in particular the 3D Landau-de Gennes model  for liquid crystals, with three distinct elastic constants.  Similar results are known in the isotropic case $W(\nabla u)=\vert \nabla u\vert^2$ and rely on three ingredients: a monotonicity formula for the scale-invariant energy on small balls, a uniform pointwise bound, and a Bochner equation for the energy density. In the level of generality we consider, all of these ingredients are absent. 
{ In particular, the lack of monotonicity formula is an important reason why optimal estimates on the singular set of $W$-harmonic maps constitute an open problem.}
Our novel argument relies on showing appropriate decay for the energy on small balls, separately at scales smaller and larger than $\e$: the former is obtained from the regularity of solutions to elliptic systems while the latter is inherited from the regularity of $W$-harmonic maps.  This also allows us to handle physically relevant boundary conditions for which, even in the isotropic case, uniform convergence up to the boundary was open.
\end{abstract}

\title{Singular perturbation of manifold-valued maps with anisotropic energy}
\date{}
\author{Andres Contreras \and Xavier Lamy}

\address[A.~Contreras]{Department of Mathematical Sciences, New Mexico State University, Las Cruces,
New Mexico, USA}
\email{acontre@nmsu.edu} 

\address[X.~Lamy]{Institut de Math\'ematiques de Toulouse; UMR 5219, Universit\'e de Toulouse; CNRS, UPS IMT, F-31062 Toulouse Cedex 9, France}
\email{xlamy@math.univ-toulouse.fr}

\maketitle

\section{Introduction}

Let $\Omega\subset\R^n$ ($n\geq 3$) be a smooth domain and $u\colon\Omega\to\R^k.$ For $\e>0$ define:
\begin{equation*}
E_\e(u;\Omega):=\int_\Omega W(x,\nabla u) + \frac{1}{\e^2}f(u).
\end{equation*}
Here
 $f\colon\R^k\to [0,\infty)$  is a smooth potential such that $\mathcal N=\lbrace f=0\rbrace$ is a smooth submanifold of $\R^k$,  with $f$ vanishing nondegenerately on $\mathcal N$, and  $W\colon\Omega\times\R^{k\times n}\to [0,\infty)$ is an elastic energy density such that $W(x,\cdot)$ is a positive definite quadratic form on $\R^{k\times n}$, uniformly in $x$.  
 
We are interested in the behavior, as $\e\to 0$, of minimizers of $E_\e$  with respect to general boundary conditions: strong or weak anchoring. In the case of strong (Dirichlet) anchoring, admissible functions have a prescribed profile at the boundary, which we assume to be regular  and to take values into $\mathcal N$; in the case of weak anchoring, one does not fix a profile at the boundary but instead one considers the modifed functional
\begin{equation}\label{eq:F_wa}
F_\e(u; \Omega):=E_\e(u;\Omega)+\int_{\partial \Omega} g(x, u),
\end{equation}
where $g:\partial\Omega\times \mathbb{R}^k\to [0,\infty)$ is a $C^2$ function.
 It can be checked (see e.g. \cite{majumdarzarnescu10}) that minimizers of $E_\e$ converge, up to subsequence, strongly in $H^1,$ to a map $u_\star\colon\Omega\to\mathcal N$ which minimizes the energy
\begin{equation*}
E_\star(u;\Omega)=\int_\Omega W(x,\nabla u),\qquad u\colon\Omega\to\mathcal N
\end{equation*}
subject to the same Dirichlet boundary conditions, in the strong anchoring case. Similarly, in the weak anchoring situation, minimizers of $F_\e$ converge to minimizers of
\begin{equation*}
F_\star(u;\Omega)=\int_\Omega W(x,\nabla u)+\int_{\partial \Omega} g(x, u),\qquad u\colon\Omega\to\mathcal N,
\end{equation*}
again, strongly in $H^1$ and up to subsequence. (In fact strong $H^1$ compactness of bounded energy sequences holds locally, without fixing boundary conditions, see Appendix~\ref{a:compact}.)  All difficulties arising in this article are already present in the case of $W$ with constant coefficients.  We are not concerned with the critical dimension $n=2$, where the issues are very different.

A strong motivation for the study of the class of energy functionals $E_\e,$  comes from their connection to physical problems in material sciences. An important ocurrence is the Landau-de Gennes energy for nematic liquid crystals \cite{mottramnewton}, where the unknown is a map $Q\colon\Omega\to\mathcal{S}_0:=\lbrace Q\in\R^{3\times 3}_{sym},\,\tr Q=0\rbrace\simeq \R^5$, and 
\begin{align}\label{WLdG}
W_{LdG} (\nabla Q) &=L_1\abs{\nabla Q}^2 + L_2\, \partial_j Q_{ik}\partial_k Q_{ij} + L_3\, \partial_j Q_{ij}\partial_k Q_{ik},
\\
\label{fLdG}
f_{LdG} (Q)  &= a^2 \abs{Q}^2 -b^2\tr(Q^3) +c^2 \abs{Q}^4.
\end{align}
The vacuum manifold is $\mathcal N=\lbrace s_\star (n\otimes n-I/3)\colon n\in\mathbb S^2\rbrace$ for some $s_\star(a,b,c)>0$.  In order for $W_{LdG}$ to be positive definite, the elastic constants satisfy (see e.g. \cite{kitavtsev16})
\begin{equation}\label{elastWpos}
L_1+L_2>0, \, 2L_1-L_2>0\text{ and } 6L_1+L_2+10 L_3>0.
\end{equation} 

 This theory has motivated a wealth of new mathematical results in the past few years, regarding e.g. the London limit $\e\to 0$  \cite{majumdarzarnescu10,canevari2d,
canevari17,baumanparkphillips12,
golovatymontero14,singperturb}, the fine structure of defects \cite{INSZinstab2d,INSZstab2d,
difrattaetal16,INSZuniqhedgehog,
INSZstabhedgehog,canevari2d,biaxialescape,henaomajumdarpisante17}, colloidal suspensions \cite{alamabronsardlamy16saturn,alamabronsardlamy16phys,alamabronsardlamy17} or lifting issues \cite{ballzarnescu08,ballzarnescu11,ballbedford15,bedford16,ignatlamy17}. The isotropic case corresponds to $L_2=L_3=0$, a restriction which was assumed in most of the above works. The general anisotropic case of three distinct elastic constants has  remained largely unexplored, due to the many mathematical challenges involved (see e.g. \cite{kitavtsev16}). In particular, our results extend the conclusions of  \cite{majumdarzarnescu10,nguyenzarnescu13,singperturb} to any $L_1$, $L_2$ and $L_3$ such that $W$ remains positive definite. Note that another, physically motivated potential $f(Q)$ was introduced in \cite{ballmajumdar10}, to which it would be interesting to extend our analysis.

In the isotropic case 
\begin{equation*}
W_{iso}(\nabla u)= \abs{\nabla u}^2\text{ (or }\abs{\nabla_{\!g} u}^2\text{ for some Riemannian metric }g\text{),}
\end{equation*}
minimizers of $E_\star$ are $\mathcal N$-valued harmonic maps. They are smooth outside a rectifiable singular set of dimension at most $(n-3)$, and the convergence of $u_\e$ towards $u_\star$ is locally uniform away from this singular set and from the boundary \cite{chenlin93,majumdarzarnescu10,singperturb}. Moreover, for fixed $\mathcal N$-valued Dirichlet boundary conditions, the convergence is also uniform up to the boundary \cite{nguyenzarnescu13,singperturb}.  For the Ginzburg-Landau potential  $k=2$, $f(u)=(1-\abs{u}^2)^2$, uniform convergence up to the boundary is also obtained for weak anchoring in the special case $g(x,u)=\abs{u-u_b(x)}^2$ in \cite{BPW18}.

For more general anisotropic elastic energies, the regularity of minimizers of $E_\star$ is not fully understood. It is known that the singular set has dimension strictly less than $n-2$ \cite{hkl88,hong04}, but due to the failure of the energy monotonicity formula, Federer's dimension reduction argument can not be applied to show that the singular set has dimension at most $n-3$. It is an open problem to find the optimal estimate on the dimension of the singular set for these anisotropic harmonic maps, but not only that, the uniform convergence away from the singular set has also proved to be an elusive question due to the technical limitations of the classical approaches, mainly derived from the theory of harmonic maps. Here we address this  open question and extend the results in \cite{singperturb} to include, for the first time, anisotropic elastic energies.

To shed some light on the  underlying difficulties, we mention that the available proofs of uniform convergence for the isotropic energy \cite{majumdarzarnescu10,nguyenzarnescu13,singperturb} follow the strategy of \cite{chenlin93}, inspired by \cite{schoen84} (related results can be found in \cite{BBH1}  for $n=2$, and \cite{BPW18} for higher dimensions, in the case of the Ginzburg-Landau potential). The main tool is  a small energy estimate, which relies on 3 crucial ingredients:
\begin{itemize}
\item a uniform $L^\infty$ bound $\norm{u_\e}_{L^\infty}\leq M$,
\item a \enquote{Bochner type} inequality $-\Delta e_\e \lesssim e_\e^2$ satisfied by $e_\e=\frac 12 \abs{\nabla u_\e}^2+\e^{-2}f(u_\e)$,
\item and a monotonicity formula for the renormalized energy $\frac{d}{dr}[r^{2-n}E_\e(u_\e;B_r)]\geq 0$.
\end{itemize} 
All three of these ingredients do not seem to be available in the anisotropic case. 
 We circumvent these difficulties by using variational arguments (in contrast with the PDE ones used in the isotropic case). Moreover, our variational approach allows us to obtain uniform estimates at the boundary for  strong and weak anchoring, while the previous methods could only deal with strong anchoring (see \cite{singperturb}). 

\subsection*{Locally uniform convergence for general anisotropic energies.}

As pointed out before, we already know that a subsequence of minimizers $u_\e$ of the functionals $E_\e$ (resp. $F_\e$) converges in $H^1$ to a generalized harmonic map $u_\star$, i.e. a minimizer of $E_\star$ (resp. $F_\star$). In order to improve this to uniform convergence away from the singular set $\mathcal{S}$  of $u_\star$, we  establish  uniform H\"{o}lder bounds for $u_\e$ on compact subsets of $\overline{\Omega}\setminus\mathcal{S}$. Classically, this is done by means of a small energy estimate:
\[r^{2\alpha}|u_\e|^2_{C^\alpha(B(x_0,r))}\leq C(W,f) r^{2-n}E_\e (u_\e; B(x_0, 2r)) , \] 
for a constant $C(W,f)$ depending on the specified parameters (but not on $\e$), provided the renormalized energy $r^{2-n}E_\e (u_\e; B(x_0, 2r))$ is small enough. Granted such estimate, $H^1$ convergence automatically improves to uniform convergence away from $\mathcal S$, since there the renormalized energy of $u_\star$ is small. Therefore we will concentrate on proving small energy estimates, in the interior and at the boundary.

For $u_\star$, an equivalent of this small energy estimate is indeed valid, and at the core of the regularity theory in \cite{hkl86,luckhaus88}. However its proof relies strongly on the scaling invariance of the energy $E_\star$. It is at this level that a big difference arises:  our perturbed energy $E_\e$ contains two terms which scale differently. This is reflected in the presence of a characteristic length scale $\e$. At scales larger than $\e$ we expect minimizers to behave like generalized harmonic maps (i.e. minimizers of $u_\star$). As we move to finer scales, the particular shape of the potential $f$ plays a more prominent role, and thus our minimizer resembles less and less a harmonic map. Note that in the isotropic case this effect can be somewhat controlled, as the monotonicity formula ensures that if the energy is small at one scale, then it is automatically small at all smaller scales.

Of the three crucial ingredients for the small energy estimate, which are present in the isotropic case but not here -- namely, the uniform pointwise bound, the Bochner equation satisfied by the energy density, and the energy monotonicity formula -- the uniform pointwise bound $\norm{u_\e}_\infty\leq M$ turns out to be the most problematic. We manage in fact to develop a general method that needs neither Bochner equation nor energy monotonicity, but, in order to avoid assuming an \textit{a priori} $L^\infty$ bound, we need to restrict ourselves to dimension $n=3$  and potentials $f$ satisfying some additional technical assumption. These restrictions are satisfied in the physically relevant case of the Landau-de Gennes functional. To the best of our knowledge, this work is the first to treat this more general model.

Specifically, we always make the following two generic assumptions:
\begin{itemize}
\item  $W(x,\nabla u)$ is a positive definite quadratic form on $\R^{k\times n}$ with coefficients depending smoothly on $x$, i.e. $W(x,\nabla u)=a_{ij}^{\alpha\beta}(x)\partial_i u^\alpha\partial_j u^\beta$ with
\begin{equation}\tag{A1}\label{A1}
\lambda \abs{\xi}^2\leq a_{ij}^{\alpha\beta}\xi_i^\alpha\xi_j^\beta\leq\Lambda \abs{\xi}^2\quad\forall\xi\in\R^{k\times n},\qquad\text{and }\norm{a_{ij}^{\alpha\beta}}_{C^1}\leq\Lambda,
\end{equation}
for some $\Lambda>\lambda>0$,
\item  $f$ vanishes nondegenerately on $\mathcal N$, i.e. $\nabla^2 f(z)$ restricted to $(T_z\mathcal N)^\perp$ is positive definite for all $z\in\mathcal N$, and $f$ does not vanish at infinity. This implies (see e.g. \cite{singperturb}) that
\begin{equation}\tag{A2}\label{A2}
f(z)\lesssim \dist^2(z,\mathcal N)\lesssim f(z)\quad\text{for }z\text{ close enough to }\mathcal N,\qquad\text{ and } \liminf_{\abs{z}\to\infty} f(z)>0.
\end{equation}
\end{itemize}
 Here and throughout the article, the symbol $\lesssim$ will denote inequality up to a multiplicative constant that depends only on the fixed parameters ($f$, $W$, $g$), unless otherwise specified.

But, in addition to these natural requirements, we will assume either
  that there exists $M>0$ such that
\begin{equation}\tag{A3a}\label{A3a}
\norm{u_\e}_\infty\leq M\qquad\forall \e>0,
\end{equation}
or  that
\begin{equation}\tag{A3b}\label{A3b}
\begin{gathered}
n=3,\text{ and there exist }p>\frac 32\text{ and }\frac 12 \leq a \leq \min\left(\frac 45,\frac 43 -\frac 1p\right),\\
\text{ such that }\abs{\nabla f(z)}\lesssim \abs{z}^{\frac 6p}\text{ and }\abs{\nabla f(z)}\lesssim f(z)^a\text{ as }\abs{z}\to\infty.
\end{gathered}
\end{equation}

\begin{rem}\label{r:A3a}
We actually expect the pointwise bound \eqref{A3a} to hold true for minimizers of $E_\e$ under rather mild conditions on $f$. In the isotropic case, one only needs to assume that $u\cdot \nabla f(u)\geq 0$ for $\abs{u}\geq M$ (in fact this is valid for all critical points that satisfy the Euler-Lagrange equations \cite{lamy14}), but establishing this estimate in the anisotropic case turns out to be surprisingly difficult. To shed more light on this issue, note that one obvious difference between isotropic and general anisotropic $W$'s is that, in the isotropic case, the linear second order elliptic operator $\mathcal L$ associated to critical points of $\int W(\nabla u)$ comes in the form of a scalar operator  acting separately on each component, while in the anisotropic case it really couples all components. It is known that, even for minimizers of $\int W(\nabla u)$ without any constraint, i.e. solutions of $\mathcal Lu=0$, the maximum principle $\norm{u}_{L^\infty(\Omega)}\leq \norm{u}_{L^\infty(\partial\Omega)}$ holds in this sharp form if and only if $\mathcal L$ is (up to a linear change of variables) isotropic \cite[Theorem~2.4]{kresinmazya}. In general one only has $\norm{u}_{L^\infty(\Omega)}\leq C \norm{u}_{L^\infty(\partial\Omega)}$ for some $C>1$. In other words the elastic term $W(\nabla u)$ does  somehow penalize large pointwise values of $u$, but not as sharply as it does in the isotropic case: therefore one can not hope to directly generalize the isotropic arguments for \eqref{A3a} to the anisotropic case.
\end{rem}

The following theorem is a corollary of our main results Theorems~\ref{t:int}, \ref{t:bdry_D} and \ref{t:bdry_WA_a}.
\begin{thm}\label{t:summary}
 Let $(u_{\e_k})$ be a sequence of  minimizers of $E_{\e_k}$ in $\Omega$ (with respect to their own boundary conditions) and assume that $u_{\e_k}$ converges strongly in $H^1$ to $u_\star\in C^\alpha(\overline\Omega\setminus\mathcal S)$. Assume \eqref{A1}-\eqref{A2}, and moreover that either \eqref{A3a} or \eqref{A3b} holds. Then we have:
\begin{itemize}
\item \emph{Interior convergence:}
\begin{equation*}
u_{\e_k}\longrightarrow u_\star\text{ locally uniformly in }\Omega\setminus \mathcal S,
\end{equation*}
\item \emph{Convergence up to the boundary for strong anchoring:} if $u_{\e_k}=u_b$ on $\partial\Omega$ for some fixed $u_b\in C^2(\partial\Omega;\mathcal N)$, then
\begin{equation*}
u_{\e_k}\longrightarrow u_\star\text{ locally uniformly in }\overline\Omega\setminus \mathcal S,
\end{equation*}
\item \emph{Convergence up to the boundary for weak anchoring:} if $u_{\e_k}$ minimizes $F_{\e_k}$ \eqref{eq:F_wa}, and \eqref{A3a} holds, then
\begin{equation*}
u_{\e_k}\longrightarrow u_\star\text{ locally uniformly in }\overline\Omega\setminus \mathcal S.
\end{equation*}
\end{itemize}
\end{thm}

In the proof of Theorem~\ref{t:summary} we overcome the lack of Bochner equation and monotonicity formula by adopting a more fundamental, but flexible approach.  Establishing a H\"{o}lder estimate amounts to proving a suitable energy decay on small balls, and we do this in two steps. On balls of radii much larger than $\e$ this decay is obtained by means of variational arguments available in the harmonic map literature: carefully constructed comparison maps lead to an energy improvement estimate, that establishes energy decay from one fixed scale to another fixed smaller scale. This may be iterated as long as the scale remains much larger than $\e$ and constitutes the first step. In the second step we deal with scales of order $\e$ and below. There, the energy decay is obtained from elliptic estimates for a fixed $\e_0$. We exploit the shape of the potential in a decisive manner to be able to connect these estimates to the ones in the first step. This is where the uniform $L^\infty$ bound \eqref{A3a} plays a crucial role, and in its absence we have to resort to the technical assumption \eqref{A3b}.

\begin{rem}\label{r:A3b} A few observations about this technical assumption  are in order.
\begin{itemize}
\item Note that under assumption \eqref{A3b}, we are able to obtain the uniform convergence away from $\mathcal S$, but we still do not know if  the uniform bound \eqref{A3a} holds.
\item Note also that in the weak anchoring case we cannot avoid assuming \eqref{A3a}. We will comment more about this in \S~\ref{ssWA}.
\item Finally, we point out that \eqref{A3b}
consists of two growth requirements on $\nabla f$ in dimension $n=3$. The first growth requirement $\abs{\nabla f(z)}\lesssim\abs{z}^{6/p}$ is fairly natural; it allows to obtain H\"older continuity of any solution $u$ of the Euler-Lagrange equations for a fixed $\e_0>0,$ via classical arguments relying on Calderon-Zygmund estimates. On the other hand, the second hypothesis $\abs{\nabla f}\lesssim f^a$ is, admittedly, less natural but it is what ultimately allows us to make the connection with the estimates at large scales obtained in the first step.  Concerning the first growth requirement, alternate hypotheses -- that also apply to the Landau-de Gennes energy -- are available, but we choose the current presentation due to its transparency.
\end{itemize}
\end{rem}

\subsection*{Landau-de Gennes with three distinct elastic constants.}

In the physically relevant case of the Landau-de Gennes potential, assumption \eqref{A3b} is satisfied with $p=2$ and $a=3/4$,  since $f_{LdG}(Q)\gtrsim\abs{Q}^4$ and $\abs{\nabla f_{LdG}(Q)}\lesssim\abs{Q}^3$ as $\abs{Q}\to\infty$. Thus Theorem~\ref{t:summary} is the first to provide an unconditional result in this context. More explicitly, specializing to the Landau-de Gennes model, Theorem \ref{t:summary} says:

\begin{thm}\label{TheoLdG} 
Let $\Omega\subseteq \R^3$ a smooth bounded domain,  $(Q_\e)_{\e>0}\subseteq H^1(\Omega;\mathcal{S}_0)$ be a family of minimizers with respect to their own boundary conditions of the Landau-de Gennes energy
\[E^{LdG}_\e (Q):= \int_\Omega \left( W_{LdG} (\nabla Q)+ \frac{1}{\e^2}f_{LdG} (Q) \right)\, dx ,\]
where $W_{LdG} $ and $f_{LdG} $ are defined in \eqref{WLdG} and \eqref{fLdG} respectively. Assume the elastic constants $L_1, L_2,$ and $L_3$ satisfy \eqref{elastWpos}. Then
\begin{itemize}
\item[a)] There exists a subsequence $\e_k\to 0,$ such that the maps $Q_{\e_k}$ converge to an $\mathcal{N}$-valued, $W_{LdG}$-harmonic map $Q_\star$ strongly in $H^1$ and locally uniformly in $\Omega\setminus \mathcal S$, where $\mathcal S$ is the singular set of $Q_\star.$ 

\item[b)] If in addition we assume that there exists a $C^2$ map $Q_b:\partial \Omega\to \mathcal{N}$ such that $Q_\e =Q_b$ on $\partial\Omega$,  
then the convergence is locally uniform up to the boundary, that is, in $\overline{\Omega}\setminus \mathcal{S}.$ 

\item[c)] Given any $C^2$ function $g:\partial \Omega\times \mathcal{S}_0 \to [0, \infty)$,  and under the additional assumption that 
\begin{equation*}
\quad\sup_{\e >0} \norm{Q_\e}_\infty <\infty,
\end{equation*}
the same conclusion holds  for minimizers of the weak anchoring energy $F^{LdG}_\e (Q):=E^{LdG}_\e (Q) +\int_{\partial \Omega} g(x, Q)$, that is, the convergence is locally uniform in $\overline\Omega\setminus\mathcal S$.
\end{itemize}
\end{thm}

The paper is organized as follows: in the next section we prove the fundamental lemmas that imply energy decay at the different scales, and use these to prove the interior small energy estimate.  In Section~\ref{s:bdry} we outline the adaptations needed to handle the boundary estimates, which follow the same general strategy as for the interior, but where technical differences make the proofs more delicate. The paper finishes with two appendices where we prove a technical boundary modification lemma, and the strong $H^1_{loc}$ compactness of bounded energy sequences.

\vspace*{1em}
\noindent\textbf{Acknowledgments.} The work of A.C. was partially supported by a grant from the
Simons Foundation $\#$ 426318. The work of X.L. was partially supported by the ANR project ANR-18-CE40-0023.
The authors wish to thank R\'emy Rodiac for many useful discussions, and Changyou Wang for pointing out a gap in the initial proof of Lemma~\ref{l:largescales}. A.C. would like to thank the Institut de Math\'ematiques de Toulouse and Universit\'e Paul Sabatier for hosting him as an invited professor which allowed for the conclusion of this work.

\section{Interior estimates}\label{s:int}

In this section we prove the interior small energy estimate.

\begin{thm}\label{t:int}
Assume that $f$ satisfies \eqref{A2}, that $W$ satisfies \eqref{A1}, and moreover that either \eqref{A3a} or \eqref{A3b} holds.
There exist $\delta,\e_0>0$ and $\alpha\in (0,1)$ (depending on $\lambda$, $\Lambda$ and $f$, and under \eqref{A3a} also on $M$) such that for any $r_0\in (0,1)$, $\e\in (0,r_0\e_0)$, and any
 $u_\e$ minimizing $E_\e(\cdot;B_{2r_0})$ with respect to its own boundary conditions,
\begin{equation*}
(2r_0)^{2-n}E_\e(u_\e;B_{2r_0}) \leq \delta^2\quad\Longrightarrow\quad r_0^{2\alpha} \abs{u_{\e_\ell}}^{2}_{C^\alpha(B_{r_0})}\lesssim  (2r_0)^{2-n}E_\e(u_\e;B_{2r_0}),
\end{equation*}
where the constant in the last inequality depends on $\lambda$, $\Lambda$ and $f$, and under \eqref{A3a} also on $M$.
\end{thm}

We prove Theorem~\ref{t:int} by obtaining uniform bounds for the decay of the energy on small balls.
Our proof reflects the fact that different scales are at stake in this problem, due to the different homogeneities of the two terms in $E_\e$. At scales larger than $\e$ the decay of the Dirichlet energy is inherited from the small energy regularity of (anisotropic) \enquote{harmonic} maps (Lemma~\ref{l:largescales}). At scales smaller than $\e$ it is inherited from regularity estimates for elliptic systems (Lemma~\ref{l:smallscales}). This is where the absence of a uniform $L^\infty$ bound (that was easy to obtain in the isotropic case) is an issue and we have to either assume it \eqref{A3a} or to require $f$ to satisfy the technical assumption \eqref{A3b}.

\subsection{Energy decay at larges scales}

We begin by { an energy improvement lemma which implies} that $u_\e$ has the { energy} decay of minimizing $W$-harmonic maps at scales larger than $\e.$ This is the only part of our argument where the minimality of $u_\e$ is crucial; at smaller scales the decay comes from the regularity of solutions to a PDE. It is worth mentioning that { energy} decay of stationary $W$-harmonic maps is not known and that mere criticality is not sufficient for this property to hold even in the isotropic case. { Thus, to obtain similar H\"older estimates for critical points of $E_\e$ one would need to avoid the contradiction argument which relies on the regularity of $W$-harmonic maps.}

\begin{lem}\label{l:largescales}
Assume \eqref{A1} and \eqref{A2}.
There exist $\delta_0,\e_0>0$ and $\theta_0\in (0,1/2)$ (depending on $\lambda$, $\Lambda$ and $f$) such that any minimizer $u_\e$ of $E_\e(\cdot;B_1)$ with $\e\in (0,\e_0)$  satisfies
\begin{equation*}
E_\e(u_\e;B_1) \leq \delta_0^2\quad\Longrightarrow \quad \theta_0^{2-n}E_\e(u_\e;B_{\theta_0}) \leq \frac 1{2} E_\e(u_\e;B_1).
\end{equation*}
\end{lem}

The proof of Lemma~\ref{l:largescales} is modelled on the proof of the corresponding energy decay result for minimizing $\mathcal N$-valued maps in \cite{hkl86,luckhaus88}. There, it relies on the observation that a sequence of minimizers with arbitrarily small energy converges, after translating and rescaling its image, to a map minimizing energy under the \emph{linear} constraint $v\in T_z\mathcal N$ a.e., for some $z\in\mathcal N$. For such a map, classical elliptic regularity applies and this allows to conclude. Here we are faced with the additional difficulty that the potential term in the energy does not behave well with respect to such rescaling. To circumvent this difficulty we first need to be able to modify boundary values with the help of the following lemma, whose proof we postpone to the appendix.

\begin{lem}\label{l:modifboundary}
There exists $\delta_1=\delta_1(\mathcal N,f)>0$ such that for all $0< \e \leq \lambda<1$ and any $u\in H^1(\partial B_1;\R^k)$ with $E_\e(u;\partial B_1)\leq\delta_1^2\lambda^{n-1}$, there exist
\begin{align*}
& w\in H^1(\partial B_1;\mathcal N),\quad\varphi\in H^1(B_1\setminus B_{1-\lambda};\R^k),\\
& \text{with }\varphi = u\text{ on }\partial B_1,\quad \varphi= w((1-\lambda)\cdot)\text{ on }\partial B_{1-\lambda},
\end{align*}
satisfying the bounds
\begin{equation*}
E_\e(\varphi;B_1\setminus B_{1-\lambda})\lesssim \lambda E_\e(u;\partial B_1)\quad\text{and}\quad \int_{\partial B_1} \abs{\nabla w}^2 \lesssim E_\e (u;\partial B_1).
\end{equation*}
\end{lem}

\begin{rem}
In the Landau-de Gennes setting ($k=5$, $\mathcal N\approx\mathbb{RP}^2$), results similar to Lemma~\ref{l:modifboundary} are proved in \cite{canevari17}. There the energy $E_\e(u)$ is allowed to be of order $o(|\log\e|)$, and $\lambda\sim\e^{1/2}\abs{\log\e}$. These results would enable us to perform the proof of Lemma~\ref{l:largescales} in the Landau-de Gennes case, but in our case we need only to consider small energies, which makes the proof of  Lemma~\ref{l:modifboundary} much simpler, and independent of the topology of $\mathcal N$.
\end{rem}

Granted this boundary modification lemma, we turn to the proof of Lemma~\ref{l:largescales}.

\begin{proof}[Proof of Lemma~\ref{l:largescales}]
Thanks to the ellipticity assumption \eqref{A1}, minimizers $v$ of $\int_{B_1}W(\nabla v)$ under the linear constraint $v\in T_z\mathcal N$ a.e. enjoy elliptic regularity estimates, uniformly in $z\in\mathcal N$ (see for example \cite[\S~10]{ADN2}). In particular, there exists $\theta_0\in (0,1/2)$ such that for any $z\in\mathcal N$ and any $v\in H^1(B_1;T_z\mathcal N)$ minimizing $\int_{B_1}W(\nabla \tilde v)$ among all maps $\tilde v\in H^1(B_1; T_z\mathcal N)$ with $\tilde v =v $ on $\partial B_1$, it holds
\begin{equation}\label{eq:regv}
\theta_0^{2-n}\int_{B_{\theta_0}}W(\nabla v) \leq \frac 14 \int_{B_{\frac 34}} W(\nabla v).
\end{equation}
In order to prove the validity of Lemma~\ref{l:largescales} for this value of $\theta_0$, we assume by contradiction that there exist sequences $\e_\ell\to 0$ and $u_\ell$ minimizing $E_{\e_\ell}(\cdot;B_1)$ such that
\begin{equation}\label{eq:assu_l}
E_{\e_\ell}(u_\ell;B_1)=:\delta_\ell^2\to 0,\qquad\text{and }\quad \theta_0^{n-2} E_{\e_\ell}(u_\ell;B_{\theta_0})>\frac 12 E_{\e_\ell}(u_\ell; B_1).
\end{equation}
By Fubini's theorem we may choose $\rho\in [3/4,1]$ such that $E_{\e_\ell}(u_\ell;\partial B_\rho)\lesssim \delta_\ell^2$. To simplify notations we assume in the sequel that $\rho=1$, i.e. we have
\begin{equation*}
E_{\e_\ell}(u_\ell;\partial B_1)\lesssim\delta_\ell^2.
\end{equation*}
The assumptions \eqref{A2} on $f$  imply that $\dist(u_\ell,\mathcal N)\to 0$ a.e., and we deduce that $u_\ell$ converges strongly in $H^1(B_1;\R^k)$ to a constant $z\in \mathcal N$. From Lemma~\ref{l:modifboundary} we obtain $\lambda_\ell\to 0$, $w_\ell\colon\partial B_1\to\mathcal N$, $\varphi_\ell\colon B_1\setminus B_{1-\lambda_\ell}\to\mathbb R^k$ such that
\begin{align}
\varphi_\ell = u_\ell\text{ on }\partial B_1,\quad\varphi_\ell = w_\ell((1-\lambda_\ell)\cdot)\text{ on }\partial B_{1-\lambda_\ell},\nonumber \\
E_{\e_\ell}(\varphi_\ell;B_1\setminus B_{1-\lambda_\ell})\lesssim \lambda_\ell \delta_\ell^2,\qquad \int_{\partial B_1}\abs{\nabla w_\ell}^2\lesssim\delta_\ell^2.\label{eq:estimw_l}
\end{align}
Note that, as argued in \cite[Corollary~34]{canevari17}, this implies that
\begin{equation}\label{eq:wclosetoubdry}
\int_{\partial B_1}\abs{w_\ell -u_\ell}^2 \lesssim \lambda_\ell \delta_\ell^2.
\end{equation}
Since $n\geq 3$, the set of maps $w\in H^1(B_1;\mathcal N)$ which agree with $w_\ell$ on $\partial B_1$ is not void (see e.g. \cite[Lemma~1.1]{hkl86}), and we can choose $\overline w_\ell\in H^1(B_1;\mathcal N)$ such that $\overline w_\ell = w_\ell$ on $\partial B_1$, and
\begin{align}
\alpha_\ell^2:=\int_{B_1} W(\nabla \overline w_\ell)=\min\left\lbrace \int_{B_1}W(\nabla w)\colon w\in H^1(B_1;\mathcal N),\,{w_{\lfloor \partial B_1}=w_\ell}\right\rbrace .\label{eq:wbar}
\end{align}
The map 
\begin{equation*}
\tilde w_\ell = \begin{cases}
\varphi_\ell & \text{ in }B_1\setminus B_{1-\lambda_\ell},\\
\overline w_\ell ((1-\lambda_\ell)\cdot)&\text{ in }B_{1-\lambda_\ell},
\end{cases}
\end{equation*}
agrees with $u_\ell$ on $\partial B_1$, an we infer, recalling \eqref{eq:estimw_l}, that
\begin{equation*}
\delta_\ell^2 =E_{\e_{\ell}}(u_\ell;B_1)\leq E_{\e_\ell}(\tilde w_\ell;B_1)\leq (1+o(1))\alpha_\ell^2 + o(1) \delta_\ell^2,
\end{equation*}
hence
\begin{equation}\label{eq:deltalessthanalpha}
\delta_\ell^2\leq (1+o(1))\alpha_\ell^2.
\end{equation}
On the other hand, the minimality property \eqref{eq:wbar} of $\overline w_\ell$ ensures, comparing its energy with the energy of the 0-homogeneous map $w_\ell(x/\abs{x})$, that
\begin{equation}\label{eq:alphalesssimdelta}
\alpha_\ell^2 \lesssim \int_{\partial B_1}\abs{\nabla w_\ell}^2 \lesssim \delta_\ell^2.
\end{equation}
This, together with \eqref{eq:wclosetoubdry} and the fact that $u_\ell$ converges to $z\in\mathcal N$, implies that $w_\ell \to z$ in $H^1(B_1;\mathcal \R^k)$.
Next we argue as in \cite{luckhaus88} and translate and rescale $\overline w_\ell$ in order to obtain a limiting map with values into $T_z\mathcal N$. Since $y_\ell:=\fint w_\ell$ converges to $z\in\mathcal N$, for large enough $\ell$ we may define $z_\ell:=\pi_{\mathcal N}(y_\ell)$, and Poincar\'e's inequality then ensures that
\begin{equation}\label{eq:wclosetoz}
\int_{B_1}\dist(\overline w_\ell -z_\ell, T_{z_\ell}\mathcal N) \lesssim \int_{B_1}\abs{\overline w_\ell -z_\ell}^2\lesssim \alpha_\ell^2.
\end{equation}
Hence the map
\begin{equation*}
v_\ell := \frac{1}{\alpha_\ell}(\overline w_\ell -z_\ell),
\end{equation*}
is bounded in $H^1(B_1;\R^k)$ and up to a subsequence (that we do not relabel) it converges weakly to a map $v\in H^1(B_1;\R^k)$ which, thanks to \eqref{eq:wclosetoz}, takes a.e. values into $T_z\mathcal N$. In particular $v_\ell$ converges strongly to $v$ in $L^2(\partial B_1;\mathbb R^k)$. Moreover $v_\ell$ is bounded in $H^1(\partial B_1;\mathbb R^k)$ thanks to \eqref{eq:estimw_l} and \eqref{eq:deltalessthanalpha}. Thus we may argue exactly as in \cite[Proposition~1]{luckhaus88} and construct good comparison maps to deduce that $v$ minimizes $\int_{B_1}W(\nabla \tilde v)$ among $T_z\mathcal N$-valued maps $\tilde v$ that agree with $v$ on $\partial B_1$, and that the convergence $v_\ell\to v$ is in fact strong in $H^1(B_1;\R^k)$. In particular, $v$ enjoys the energy decay property \eqref{eq:regv}. Thanks to the strong convergence $v_\ell\to v$ in $H^1$, after rescaling it holds
\begin{equation}\label{eq:energydecayw}
\theta_0^{2-n}\int_{B_{\theta_0}}W(\nabla \overline w_\ell)\leq \frac {1+o(1)}{4} \int_{B_1}W(\nabla\overline w_\ell).
\end{equation}

The last part of the proof consists in obtaining, from this energy decay for $\overline w_\ell$, similar energy decay for $u_\ell$, thus contradicting \eqref{eq:assu_l}. To that end we define
\begin{equation*}
\tilde v_\ell := \frac{1}{\alpha_\ell}(u_\ell -z_\ell).
\end{equation*}
Note that here we do not divide by $\delta_\ell$ but rather by the (possibly) larger $\alpha_\ell$ correspoinding to a minimizer among $\mathcal{N}$-valued maps of a closely related problem. The reason is that even though we already know that $\delta_\ell$ and $\alpha_\ell$ are of the same order, we need exact bounds in \eqref{eq:wclosetoustrong} { below} for the corresponding cancellations to take place.

Thanks to \eqref{eq:deltalessthanalpha} it holds $\int_{B_1}\abs{\nabla \tilde v_\ell}^2\lesssim 1$, and therefore $\tilde v_\ell -\fint\tilde v_\ell$ converges (up to a subsequence) weakly in $H^1(B_1;\R^k)$, and strongly in $L^2(\partial B_1;\R^k)$. 
From \eqref{eq:wclosetoubdry} and \eqref{eq:deltalessthanalpha} we infer that
\begin{equation*}
\int_{\partial B_1}\abs{\tilde v_\ell -v_\ell}^2 \lesssim \lambda_\ell \to 0.
\end{equation*}
In particular $\tilde v_\ell$ is bounded in $L^2(\partial B_1;R^k)$, so that $\fint \tilde v_\ell$ must be bounded. We may assume that it converges in $\R^k$, and hence deduce that 
\begin{equation*}
\tilde v_\ell \longrightarrow \tilde v\quad\text{weakly in }H^1(B_1;\R^k)\text{ and a.e.},
\end{equation*}
for some map $\tilde v\in H^1(B_1;\R^k)$ such that $\tilde v=v$ on $\partial B_1$. We claim that $\tilde v\in T_z\mathcal N$ a.e. This follows from the energy bound $E_\e(u_\ell;B_1)=\delta_\ell^2\lesssim\alpha_\ell^2$, which thanks to assumption \eqref{A2} on $f$ implies in particular that
\begin{equation*}
\frac{1}{\alpha_\ell}\dist(u_\ell,\mathcal N) \longrightarrow 0\text{ a.e.}
\end{equation*}
Fixing $x\in B_1$ at which this convergence and $\tilde v_\ell(x)\to\tilde v(x)$ hold, we have
\begin{align*}
\dist(u_\ell(x),\mathcal N) &=\dist(z_\ell + \alpha_\ell \tilde v(x) + \alpha_\ell (\tilde v_\ell(x) -\tilde v(x)),\mathcal N)\\
&\geq \dist(z_\ell + \alpha_\ell\tilde v(x),\mathcal N)-o(\alpha_\ell) \\
&\geq \alpha_\ell \abs{P_\ell\tilde v(x)} -O(\alpha_\ell^2)-o(\alpha_\ell),
\end{align*}
where $P_\ell$ denotes the orthogonal projection onto $(T_{z_\ell}\mathcal N)^\perp$. This shows that $\abs{P_\ell\tilde v(x)}\to 0$ and therefore $\tilde v(x)\in T_z\mathcal N$ since $z_\ell\to z$. 

We denote by $B(\cdot,\cdot)$ the symmetric bilinear form on $\R^{k\times n}$ corresponding to $W$, so that it holds
\begin{align}
\int_{B_1} W(\nabla \overline w_\ell -\nabla u_\ell) + \frac{1}{\e_\ell^2}\int_{B_1}f(u_\ell)&=\int_{B_1}W(\nabla \overline w_\ell) + E_\e(u_\ell;B_1) -2\int_{B_1}B(\nabla \overline w_\ell,\nabla u_\ell)\nonumber \\
&\leq \alpha_\ell^2+ (1+o(1))\alpha_\ell^2 -2\alpha_\ell^2 \int_{B_1} B(\nabla v_\ell,\nabla \tilde v_\ell).\label{eq:wclosetoustrong}
\end{align}
Because $v_\ell \to v$ strongly and $\tilde v_\ell\to\tilde v$ weakly we have
\begin{align*}
\int_{B_1}B(\nabla v_\ell,\nabla\tilde v_\ell)&\to \int_{B_1}B(\nabla v,\nabla\tilde v) = \int_{B_1}B(\nabla v,\nabla(\tilde v -v)) + \int_{B_1}W(\nabla v).
\end{align*}
Since $v$ is a minimizer of $\int_{B_1} W(\nabla v)$ among $T_z\mathcal N$-valued maps and $(\tilde v-v)$ is $T_z\mathcal N$-valued and vanishes on $\partial B_1$, it holds
\begin{equation*}
\int_{B_1}B(\nabla v,\nabla(\tilde v -v))=0.
\end{equation*}
Moreover, since $\int_{B_1}W(\nabla v_\ell)=1$ and $v_\ell$ converges strongly to $v$ it holds $\int_{B_1}W(\nabla v)=1$, and therefore 
\begin{equation*}
\int_{B_1}B(\nabla v_\ell,\nabla\tilde v_\ell)=1+o(1).
\end{equation*}
Plugging this into \eqref{eq:wclosetoustrong} we find
\begin{equation*}
\int_{B_1}W(\nabla \overline w_\ell -\nabla u_\ell)+\frac{1}{\e_\ell^2}\int_{B_1}f(u_\ell) \leq o(1)\,\alpha_\ell^2.
\end{equation*}
Using this estimate and \eqref{eq:energydecayw} we obtain
\begin{align*}
\theta_0^{2-n}E_{\e_\ell}(u_\ell;B_{\theta_0})&\leq \frac 87 \theta_0^{2-n}\int_{B_{\theta_0}}W(\nabla \overline w_\ell) + o(1)\alpha_\ell^2 \\
&\leq (1+o(1))\frac 27 \int_{B_1}W(\nabla \overline w_\ell) + o(1)\alpha_\ell^2 \\
& \leq \frac{1+o(1)}{3}\int_{B_1}W(\nabla u_\ell) +o(1)\alpha_\ell^2.
\end{align*}
Since $\alpha_\ell^2\lesssim\delta_\ell^2 =E_{\e_\ell}(u_\ell;B_1)$, this contradicts \eqref{eq:assu_l} and concludes the proof of Lemma~\ref{l:largescales}.
\end{proof}

\subsection{Regularity at small scales}

\begin{lem}\label{l:smallscales}
Assume \eqref{A1}-\eqref{A2}, and either \eqref{A3a} or \eqref{A3b}.
For all $\e>0$ there exists $\delta,\alpha,C>0$ (depending on $n$, $\lambda$, $\Lambda$, $f$ and $M$, but also on $\e$) such that any minimizer $u_\e$ of $E_{\e}(\cdot;B_2)$  satisfies
\begin{equation*}
E_{\e}(u_\e;B_2) \leq \delta^2\quad\Longrightarrow \quad   r^{2-n}\int_{B_r}\abs{\nabla u_\e}^2\leq C r^{\alpha} E_{\e}(u_\e;B_2)\quad\forall r\in (0,1).
\end{equation*}
\end{lem}
\begin{rem} 
We prove in fact Lemma~\ref{l:smallscales} with $\delta=1$. We state it in this way to emphasize that this statement  is strong  enough to prove Theorem~\ref{t:int}. We expect Lemma~\ref{l:smallscales} to hold under much less restrictive assumptions: in any dimension $n\geq 3$, for an elastic energy $W(x,u,\nabla u)$ which may be a positive definite quadratic form in $\nabla u$ with coefficients depending smoothly on $x$ and $u$, and for a potential $f$ with some radial growth at infinity. 
\end{rem}

We will prove Lemma~\ref{l:smallscales} separately under the assumptions \eqref{A3a} and \eqref{A3b}. In both cases we use elliptic estimates for the equation
\begin{equation*}
\mathcal L u=\frac{1}{\e^2}\nabla f(u)
\end{equation*}
where $\mathcal L$ is the second order elliptic operator such that $\int W(\nabla\varphi)=\int \mathcal L\varphi\cdot\varphi$ for all test functions $\varphi$. This operator satisfies elliptic estimates
\begin{equation*}
\norm{\nabla^2 v}_{L^p}\lesssim \norm{\mathcal L v}_{L^p}+\norm{v}_{L^p},
\end{equation*}
for $1<p<\infty$ and all maps $v$ with compact support in the unit ball $B_1$. The inequality is up to a constant depending on $\lambda$, $\Lambda$, $n$ and $p$ (see e.g. \cite[\S~10]{ADN2} or \cite[\S~6.4]{morrey}). For any $0<R_1<R_2\leq 1$, one may apply this to $v=(u-\xi)\varphi$ for any $\xi\in\R^k$ and $\varphi$ a cut-off function such that $\varphi\equiv 1$ in $B_{R_1}$, $\varphi\equiv 0$ outside of $B_{R_2}$, and $\abs{\nabla^\ell\varphi}\lesssim (R_2-R_1)^{-\ell}$, and conclude with Poincar\'e's inequality that
\begin{equation}\label{eq:ellipt_estim}
\norm{\nabla^2 u}_{L^p(B_{R_1})}\lesssim \norm{\mathcal L u}_{L^p(B_{R_2})}+\frac{1}{(R_2-R_1)^2}\norm{\nabla u}_{L^p(B_{R_2})}.
\end{equation}
Under the uniform $L^\infty$ bound assumption \eqref{A3a}, Lemma~\ref{l:smallscales} will follow from using \eqref{A2} to bound $\mathcal L u_\e$ in terms of the energy $E_\e$, and bootstrapping the elliptic estimate \eqref{eq:ellipt_estim}. Without  the uniform $L^\infty$ bound however, both estimating $\mathcal L u_\e$ in terms of $E_\e$, and bootstrapping, do not work directly and this is why we need assumption \eqref{A3b}. 

\begin{proof}[Proof of Lemma~\ref{l:smallscales} under \eqref{A3a}]
 In this proof we drop the subscripts $\e$ to simplify notation. We fix a sequence of radii $R_k\in (1/2,1]$ by letting $R_0=1$ and $R_{k}=R_{k-1}-2^{-(k+1)}$ for $k\geq 1$.

The map $u$ solves
\begin{equation*}
\mathcal L u = \frac{1}{\e^2}\nabla f( u).
\end{equation*}
We start by applying \eqref{eq:ellipt_estim} with $p=p_0=2$, hence
\begin{equation*}
\norm{\nabla^2u}_{L^2(B_{R_1})}\lesssim  (1+\e^{-2}) \norm{\nabla f(u)}_{L^2(B_{R_0})} + \norm{\nabla u}_{L^2(B_{R_0})}.
\end{equation*}
Thanks to \eqref{A2}  and \eqref{A3a} we have $\abs{\nabla f(u)}\lesssim f(u)^{1/2}$, so that we deduce that
\begin{equation*}
\norm{\nabla^2u}_{L^2(B_{R_1})}\lesssim   \e^{-2} E(u;B_1)^{\frac 12}.
\end{equation*}
Next we set $p_1=2^*=2n/(n-2)$. By Sobolev embedding and the above we obtain
\begin{align*}
\norm{\nabla u}_{L^{p_1}(B_{R_1})}&\lesssim \norm{\nabla^2 u}_{L^2(B_{R_1})}+\norm{\nabla u}_{L^2(B_{R_1})}\lesssim   (1+\e^{-2})E(u;B_1)^{\frac 12},\\
\norm{\nabla f(u)}_{L^{p_1}(B_{R_1})}&\lesssim \norm{\nabla^2f(u)}_{L^\infty}\norm{\nabla u}_{L^2(B_{R_1})}+\norm{\nabla f(u)}_{L^2(B_{R_1})}\lesssim  (1+\e^{-2}) E(u;B_1)^{\frac 12}.
\end{align*}
Applying \eqref{eq:ellipt_estim} again we have therefore
\begin{equation*}
\norm{\nabla^2 u}_{L^{p_1}(B_{R_2})}\lesssim   (1+\e^{-4})E(u;B_1)^{\frac 12}.
\end{equation*}
If $p_1<n$, we iterate the above, and obtain a sequence of exponents $p_k<n$ such that $p_{k+1}=p_k^* = p_{k}n/(n-p_k)$, and
\begin{align*}
\norm{\nabla u}_{L^{p_k}(B_{R_k})}+ 
\norm{\nabla f(u)}_{L^{p_k}(B_{R_k})} &\lesssim  (1+\e^{-2k})E(u;B_1)^{\frac 12},\\
\norm{\nabla^2 u}_{L^{p_{k+1}}(B_{R_{k+1}})} &\lesssim   (1+\e^{-2k-2}) E(u;B_1)^{\frac 12}.
\end{align*}
The sequence $(p_k)$ is strictly increasing and  $p_{k+1}-p_k=p_k(n/(n-p_k)-1)>p_0(n/(n-p_0)-1)$, so that after a finite number of iterations we have $p_{k+1}\geq n$. If $p_{k+1}>n$, by Sobolev embedding this implies $\norm{\nabla u}_{L^\infty(B_{R_{k+1}})}\lesssim E(u;B_1)^{1/2}$. If $p_{k+1}=n$ we may replace it by $\tilde p_{k+1}<n$ but arbitrarily close to $n$ and iterate one more time. In any case we infer
\begin{equation*}
\norm{\nabla u}_{L^\infty(B_{1/2})}\lesssim   (1+\e^{-2\kappa}) E(u;B_1)^{1/2},
\end{equation*}
for some $\kappa=\kappa(n)\in\mathbb N$,
and this implies the conclusion of Lemma~\ref{l:smallscales} with $\alpha=2$.
\end{proof}

\begin{proof}[Proof of Lemma~\ref{l:smallscales} under \eqref{A3b}]   In this proof we drop the subscripts $\e$ to simplify notation.

For the convenience of the reader we recall here  assumption \eqref{A3b}:
\begin{gather*}
n=3,\text{ and there exist }p>\frac 32\text{ and }\frac 12 \leq a \leq \min\left(\frac 45,\frac 43 -\frac 1p\right),\\
\text{ such that }\abs{\nabla f(z)}\lesssim \abs{z}^{\frac 6p}\text{ and }\abs{\nabla f(z)}\lesssim f(z)^a\text{ as }\abs{z}\to\infty.
\end{gather*}
There is no loss of generality in assuming that 
\begin{equation}\label{eq:p_wlog}
\frac 32 < p\leq \frac {15} 8,
\end{equation}
since for $p\geq 15/8$ we have $4/3-1/p\geq 4/5$.

We let $\delta=1$, and consider a map $u$ minimizing $E(\cdot;B_2)$ and satisfying $E(u;B_2)\leq 1$. The map $u$ solves
\begin{equation*}
\mathcal L  u = \frac{1}{\e^2}\nabla f(u).
\end{equation*}
Since, by Sobolev embedding $u\in L^6(B_2)$, we deduce from the first growth assumption $\abs{\nabla f(z)}\lesssim \abs{z}^{6/p}$ in \eqref{A3b} that $\nabla f(u)\in L^p(B_2)$. Hence applying the elliptic estimates \eqref{eq:ellipt_estim} yields that $\nabla^2 u\in L^p_{loc}(B_2)$. Again by Sobolev embedding, we deduce that $\nabla u\in L^{p_*}_{loc}(B_2)$, where $p_*=3p/(3-p)$. Note that the condition $p>3/2$ implies that $p_*>3$. This is enough to deduce that
\begin{equation*}
\frac 1r \int_{B_r}\abs{\nabla u}^2\lesssim r^\alpha\quad\text{ for some }\alpha\in (0,1).
\end{equation*}
However  Lemma~\ref{l:smallscales} claims a bound in terms of the energy $E(u;B_2)$, which is not provided by the above argument. 
This is why we need the second growth assumption in \eqref{A3b}, namely $\abs{\nabla f}\lesssim f^a$.

Note that since $f\geq 0$, and since the nondegeneracy assumption \eqref{A2} implies $\abs{\nabla f(z)}\lesssim f(z)^{1/2}$ for $\abs{z}\lesssim 1$, using \eqref{A3b} we deduce
\begin{equation*}
\abs{\nabla f(z)}\lesssim f(z)^{1/2} + f(z)^a\qquad\forall z\in\R^k.
\end{equation*}
{ Because $1/2\leq a\leq 4/3-1/p$, this implies
\begin{equation*}
\abs{\nabla f(z)}\lesssim f(z)^{1/2} + f(z)^A\qquad\forall z\in\R^k,\; A:=\frac 43 - \frac 1p.
\end{equation*}}
Applying \eqref{eq:ellipt_estim} we obtain, for any $1\leq R_1<R_2\leq 2$,
\begin{equation*}
\int_{B_{R_1}}\abs{\nabla^2 u}^p\lesssim \frac{1}{(R_2-R_1)^{2p}}\int_{B_{R_2}}\abs{\nabla  u}^p + \e^{-2p}\int_{B_{R_2}}\abs{\nabla f( u)}^p.
\end{equation*}
On the other hand, by Sobolev embedding $W^{1,p}\subset L^{p_*}$, it holds
\begin{equation*}
\left(\int_{B_{R_1}}\abs{\nabla u}^{p_*}\, dx\right)^{\frac{2}{p_*}}\lesssim  \left(\int_{B_{R_1}}\abs{\nabla u}^p\,dx\right)^{\frac 2p} +\left(\int_{B_{R_1}}\abs{\nabla^2 u}^p\, dx\right)^{\frac 2p}.
\end{equation*}
Gathering the above, we have
\begin{align*}
\left(\int_{B_{R_1}}\abs{\nabla  u}^{p_*}\, dx\right)^{\frac 2{p_*}}&\lesssim 
\frac{1}{(R_2-R_1)^4}\left(\int_{B_{R_2}}\abs{\nabla  u}^p\, dx\right)^{\frac 2p}
+ \e^{-4}\left(\int_{B_{R_2}}\abs{\nabla f( u)}^p\, dx\right)^{\frac 2p}  \\
&\lesssim \frac{1}{(R_2-R_1)^4}\left(\int_{B_{R_2}}\abs{\nabla  u}^p\, dx\right)^{\frac 2p}\\
&\quad
+\e^{-4}\left(\int_{B_{R_2}}f( u)^{p/2}\,dx\right)^{\frac 2p} 
+\e^{-4} \left(\int_{B_{R_2}}f( u)^{{ A} p}\, dx\right)^{\frac 2p}, 
\end{align*}
and by Jensen's inequality, since $2/p\geq 1$,
\begin{align}
\left(\int_{B_{R_1}}\abs{\nabla  u}^{p_*}\, dx\right)^{\frac 2{p_*}}&\lesssim
\frac{1}{(R_2-R_1)^4}\int_{B_{R_2}}\abs{\nabla  u}^2\, dx
+\e^{-4}\int_{B_{R_2}}f( u)\,dx \\
&\quad
+\e^{-4} \left(\int_{B_{R_2}}f( u)^{{ A} p}\, dx\right)^{\frac 2p}\nonumber\\
&\lesssim \frac{1+\e^{-2}}{(R_2-R_1)^4}E(u;B_{R_2})
+ \e^{-4}\left(\int_{B_{R_2}}f( u)^{{ A} p}\, dx\right)^{\frac 2p}\label{eq:estimgradup*}
\end{align}
{ Then} we use the fact that 
\begin{equation*}
{ Ap=\frac 43 p - 1\in \left[1,\frac 32\right],}
\end{equation*}
so that by Sobolev embedding $W^{1,1}\subset L^{{A} p}$.  Hence we find
\begin{align*}
\left(\int_{B_{R_2}}f( u)^{{A} p}\, dx \right)^{\frac 2p}&\lesssim \left(\int_{B_{R_2}}f(u)\, dx\right)^{2{A}}+\left(\int_{B_{R_2}}\abs{\nabla  u}\abs{\nabla f( u)}\, dx \right)^{2{A}}\\
&\lesssim \e^{4A} E(u;B_{R_2})^{2{A}}\\
&\quad +\left(\int_{B_{R_2}}\abs{\nabla  u}f(u)^{1/2}\, dx \right)^{2{A}}
+\left(\int_{B_{R_2}}\abs{\nabla  u}f(u)^{{A}}\, dx \right)^{2{A}}.
\end{align*}
The second term in the above right-hand side is $\lesssim \e^{2A}E(u;B_{R_2})^{2{A}}$. To estimate the third term we apply H\"{o}lder's inequality to see that
\begin{align*}
\int_{B_{R_2}}\abs{\nabla  u}f(u)^{{A}}\, dx 
\lesssim \left( \int_{B_{R_2}}\abs{\nabla u}^{p_*}\, dx\right)^{\frac{1}{p_*}}\left(\int_{B_{R_2}}f(u)^{{A} p_*/(p_*-1)}\right)^{\frac{p_*-1}{p_*}}.
\end{align*}
Finally, since $E(u;B_{R_2})\leq 1$ (and $2{A}\geq 1$), we conclude
\begin{align*}
\left(\int_{B_{R_2}}f( u)^{{A} p}\, dx \right)^{\frac 2p}
&\lesssim (\e^{2A}+\e^{4A})E(u;B_{R_2}) \\
&\quad + \left( \int_{B_{R_2}}\abs{\nabla u}^{p_*}\, dx\right)^{\frac{2{A}}{p_*}}\left(\int_{B_{R_2}}f(u)^{{A} p_*/(p_*-1)}\right)^{2{A}\frac{p_*-1}{p_*}}.
\end{align*}
Recall that
\begin{align*}
A= \frac 43 -\frac 1p = \frac{p_* -1}{p_*},
\end{align*}
so that the above implies
\begin{align*}
\left(\int_{B_{R_2}}f( u)^{{A} p}\, dx \right)^{\frac 2p}
&\lesssim (\e^{2A}+\e^{4A})E(u;B_{R_2}) \\
&\quad +\e^{4A^2} \left( \int_{B_{R_2}}\abs{\nabla u}^{p_*}\, dx\right)^{\frac{2{A}}{p_*}}E(u;B_{R_2})^{2{A}^2}.
\end{align*}
Since ${A}< 1$ we may invoke Young's inequality $xy\lesssim x^{1/{A}} + y^{1/(1-{A})}$ for all $x,y\geq 0$, and deduce
\begin{align*}
\left(\int_{B_{R_2}}f( u)^{ A p}\, dx \right)^{\frac 2p}
& \lesssim (\e^{2A}+\e^{4A}) E(u;B_{R_2}) \\
&\quad + \eta^{1/{A}}\left( \int_{B_{R_2}}\abs{\nabla u}^{p_*}\, dx\right)^{\frac{2}{p_*}}+\frac{\e^{\frac {4A^2}{1-1}}}{\eta^{1/(1-{A})}}E(u;B_{R_2})^{\frac{2{A}^2}{1-{A}}},
\end{align*}
for any $\eta>0$. Since $2{A}^2/(1-{A})\geq 1$ (because ${A}\geq 1/2$) and $E(u;B_{R_2})\leq 1$ this implies
\begin{equation*}
\left(\int_{B_{R_2}}f( u)^{{A} p}\, dx \right)^{\frac 2p}
 \lesssim \frac{\e^{2A}+\e^{\frac{4A^2}{1-A}}}{\eta^{1/(1-{A})}} E(u;B_{R_2}) + \eta^{1/{A}}\left( \int_{B_{R_2}}\abs{\nabla u}^{p_*}\, dx\right)^{\frac{2}{p_*}},
\end{equation*}
and plugging this into \eqref{eq:estimgradup*} we get
\begin{align*}
\left(\int_{B_{R_1}}\abs{\nabla  u}^{p_*}\, dx\right)^{\frac 2{p_*}}
& \lesssim \left(\frac{1+\e^{-2}}{(R_2-R_1)^{4}}+\frac{\e^{2A-4}+ \e^{4\frac{A^2+A-1}{1-A}}}{\eta^{1/(1-{A})}}\right)E(u;B_{R_2}) \\
&\quad + \eta^{1/{A}}\left(\int_{B_{R_2}}\abs{\nabla u}^{p_*}\, dx\right)^{\frac 2{p_*}}.
\end{align*}
Choosing $\eta$ small enough, we infer
\begin{equation*}
\left(\int_{B_{R_1}}\abs{\nabla  u}^{p_*}\, dx\right)^{\frac 2{p_*}}
\leq\frac 12 \left(\int_{B_{R_2}}\abs{\nabla  u}^{p_*}\, dx\right)^{\frac 2{p_*}} + \frac{\e^{2A-4}+ \e^{4\frac{A^2+A-1}{1-A}}}{(R_2-R_1)^{ 4}}E(u;B_{R_2}) ,
\end{equation*}
for some constant $C>0$. Setting $\rho_j=3/2-1/(2K^{j})$ for some $K\in (1, 2^{1/ 4})$ and iterating the above estimate applied to $R_1=\rho_j$ and $R_2=\rho_{j+1}$ we obtain
\begin{align*}
\left(\int_{B_{\rho_0}}\abs{\nabla u}^{p_*}\, dx\right)^{\frac 2{p_*}}
&\leq \frac 1{2^j} \left(\int_{B_{\rho_j}}\abs{\nabla u}^{p_*}\, dx\right)^{\frac 2{p_*}} \\
&\quad + C \: (\e^{2A-4}+ \e^{4\frac{A^2+A-1}{1-A}}) \left( \frac{K-1}{2K}\right)^{2p} \left(\sum_{\ell=0}^{j-1}\left(\frac{K^{2p}}{2}\right)^\ell \right)E( u;B_{\rho_j})\\
&\leq \frac 1{2^j}\left(\int_{B_{3/2}}\abs{\nabla u}^{p_*}\, dx\right)^{\frac 2{p_*}}+ C'\: (\e^{2A-4}+ \e^{4\frac{A^2+A-1}{1-A}}) E(u;B_{3/2}).
\end{align*}
Letting $j\to\infty$ and recalling that we already know that $\nabla u\in L^{p_*}(B_{3/2})$, we deduce that
\begin{equation*}
\left(\int_{B_1}\abs{\nabla u}^{p_*}\, dx\right)^{\frac 2{p_*}}\lesssim (\e^{2A-4}+ \e^{4\frac{A^2+A-1}{1-A}}) E( u;B_{3/2}).
\end{equation*}
It is directly checked that $p_*>3$, Hence H\"older's inequality
 implies
\begin{align*}
r^{-1}\int_{B_r}\abs{\nabla u}^2 &\leq r^{2(1-3/p_*)}\left(\int_{B_r}\abs{\nabla u}^{p_*}\right)^{\frac 2{p_*}}\\
&\lesssim (\e^{2A-4}+ \e^{4\frac{A^2+A-1}{1-A}}) \: r^{2(1-3/p_*)} E(u; B_2), 
\end{align*}
for all $r\in (0,1)$.
\end{proof}

\subsection{Interior regularity}\label{ss:proofinterior}

We are ready to prove our main result.

\begin{proof}[Proof of Theorem~\ref{t:int}]
In this proof we are going to rescale repeatedly in the $x$ variable, and should accordingly define a new quadratic form $W$ at each step, unless $W$ has constant coefficients. The quadratic form $W$ with rescaled coefficients will still satisfy \eqref{A1}, hence  it will not affect the implicit constants in the conclusions of Lemmas~\ref{l:largescales} and \ref{l:smallscales}, which are uniform with respect to quadratic forms satisfying \eqref{A1}.
 Therefore we will, for the sake of clarity, assume that $W$ has constant coefficients: this does not change the strategy of the proof, but it does simplify a lot the notations  (since it avoids redefining $W$ at each step). 

Let $u_\e$ minimize $E_\e(\cdot;B_{2r_0})$ and satisfy
\begin{equation*}
(2r_0)^{2-n}E_\e(u_\e; B_{2r_0})\leq \delta^2,
\end{equation*}
for some $\delta\in (0,1]$ to be fixed later. Let $\theta_0$ and $\delta_0$ be as in Lemma~\ref{l:largescales}.

Fixing $x_0\in B_{r_0}$ and setting $\bar u(\bar x)=u_{\e}(x_0+r_0 \bar x)$ we have that $\bar u$ minimizes $E_{\bar \e}(\cdot; B_1)$ for $\bar \e =\e/r_0<\e_0$, and
\begin{equation*}
E_{\bar \e}(\bar u; B_1)\leq \delta^2.
\end{equation*}
Hence we are in a situation to apply Lemma~\ref{l:largescales} which implies that $\tilde u(\tilde x) = \bar u({\theta_0}\tilde x)$ satisfies, with $\tilde\e=\theta_0^{-1}\bar\e$,
\begin{align*}
 E_{\tilde\e}(\tilde u;B_1)&=(\theta_0)^{2-n}E_{\bar\e}(\bar u; B_{{\theta_0}})\leq \frac 12  E_{\bar\e}(\bar u;B_1)\leq \delta_0^2.
\end{align*}
By induction we may in fact apply Lemma~\ref{l:largescales} to $\tilde u(\tilde x)=\bar u (\theta_0^{j+1}\tilde x)$ and $\tilde\e =\theta_0^{-j-1}\bar \e$ for all $j\in\mathbb N$ such that $\theta_0^j >\bar\e/\e_0$ and infer
\begin{equation*}
(\theta_0^{j+1})^{2-n} E_{\bar\e}(\bar u; B_{(\theta_0)^{j+1}})\leq \left(\frac 12\right)^{j+1}  E_{\bar\e}(\bar u;B_1).
\end{equation*}
This implies
\begin{equation}\label{eq:largescaledecay}
\bar r^{2-n}E_{\bar\e}(\bar u;B_{\bar r}) \lesssim \bar r^{\alpha_0}  E_{\bar\e}(\bar u;B_1)\qquad\forall \bar r\in [\bar\e/\e_0,1),
\end{equation}
where $\alpha_0=\ln 2/\ln((\theta_0)^{-1})>0$. 

Next we set 

\begin{equation}\label{ubar} 
r_1=\bar\e/\e_0\mbox{ and }\hat u (\hat x)=\bar u(r_1 \hat x),
\end{equation}
so that $\hat u$ minimizes $E_{\hat\e}(\cdot;B_2)$ for $\hat\e=\bar\e/r_1=\e_0$,
and
\begin{equation*}
E_{\hat\e}(\hat u ; B_2)=r_1^{2-n} E_{\bar\e}(\bar u; B_{2r_1})\lesssim r_1^{\alpha_0} E_{\bar\e}(\bar u;B_1)\lesssim \delta^2.
\end{equation*}
Lemma~\ref{l:smallscales} ensures that if 
$\delta$ is small enough (depending on $n$, $\lambda$, $\Lambda$, $f$ -- and $\e_0$ which depends itself only on $n$, $\lambda$, $\Lambda$ and $f$) there exists $\hat\alpha>0$ such that
\begin{equation*}
\hat r^{2-n}\int_{B_{\hat r}}\abs{\nabla \hat u}^2\lesssim \hat r^{\hat \alpha} E_{\hat\e}(\hat u; B_2)\qquad\forall \hat r\in (0,1).
\end{equation*}
Set $\alpha=\min(\hat \alpha,\alpha_0).$ Recalling \eqref{ubar}, we can apply the previous inequality, with $\hat{r}=\bar{r}/r_1,$ to obtain
\begin{equation*}
\bar r^{2-n}\int_{B_{\bar r}}\abs{\nabla\bar u}^2 \lesssim \left(\frac{\bar r}{r_1}\right)^\alpha E_{\hat\e}(\hat u; B_2)\lesssim \bar r^\alpha  E_{\bar\e}(\bar u;B_1)\qquad\forall \bar r\in (0,r_1).
\end{equation*}
By the Campanato-Morrey characterization of H\"older spaces this implies
\begin{equation*}
r_0^{\alpha}\abs{u_\e}^2_{C^{\alpha/2}(B_{r_0})}
\lesssim (2r_0)^{2-n}E_\e(u_\e;B_{2r_0}).
\end{equation*}
\end{proof}

\section{Boundary estimates}\label{s:bdry}

In this section, we extend our previous results in the interior to both strong and weak anchoring settings. As it will become apparent, there are technical differences between the two cases; while the proof for strong anchoring works basically along the same lines as the interior case, under either assumption \eqref{A3a} or \eqref{A3b}, a proof not relying on \eqref{A3a} is, at the moment, out of reach for weak anchoring. The reason for this is that we are unable to modify Luckhaus' construction in a way that allows us to control the boundary term. We elaborate more on this in \S~\ref{ssWA}.

\subsection{The strong anchoring case}\label{ssSA}

If the boundary $\partial\Omega$ is of class $C^2$, we can cover it with small balls where it can be flattened, and after rescaling we are led to defining modified energy functionals of the form
\begin{equation*}
F_\e(u;B_2^+)=\int_{B_2^+}\left( W(x,u)+\frac{1}{\e^2}f(u)\right)a(x)\,dx,
\end{equation*}
where $B_2^+$ denotes the half ball $B_2\cap\lbrace x_n>0\rbrace\subset\R^n$, the quadratic form $W$ satisfies \eqref{A1}, and the weight $a(x)$ satisfies
\begin{equation}\tag{B1}\label{B1}
\norm{1-a}_{C^1}\leq \frac 12.
\end{equation}
We will denote by $F_\star$ the corresponding limiting energy functional for $\mathcal N$-valued maps.

To obtain boundary estimates for the original energy on $\Omega$ it suffices to consider maps $u_\e$ which minimize $F_\e$ in $B_2^+$ among maps $u$ such that $u=u_\e$ on $(\partial B_2)^+:=\partial B_2\cap\lbrace x_n>0\rbrace$ and satisfying fixed Dirichlet conditions 
\begin{equation*}
u=u_b \quad\text{ on }B'_2:=B_2\cap\lbrace x_n=0\rbrace,
\end{equation*}
for some $\mathcal N$-valued map $u_b$ of $C^2$ regularity.

\begin{thm}\label{t:bdry_D}Assume that $W$ satisfies \eqref{A1} and $a$ satisfies \eqref{B1}. Also, assume that $f$ satisfies \eqref{A2}, and that either \eqref{A3a} or \eqref{A3b} holds.
Then, there exist $\delta,\e_0>0$ and $\alpha\in (0,1)$ (depending on $n$, $\lambda$, $\Lambda$, $f$ and $M$) such that for any $r_0\in (0,1)$, $\e\in (0,r_0\e_0)$, and any
 $u_\e$ minimizing $F_\e(\cdot;B^+_{2r_0})$ with respect to its own boundary conditions and with $u_\e=u_b$ on $B'_{2r_0}$,
\begin{align*}
&(2r_0)^{2-n}F_\e(u_\e;B^+_{2r_0})+N(u_b;B'_{2r_0}) \leq \delta^2\quad\\
&\Longrightarrow\quad r_0^{2\alpha} \abs{u_{\e_\ell}}^2_{C^\alpha(B^+_{r_0})}\lesssim  (2r_0)^{2-n}F_\e(u_\e;B^+_{2r_0})+N(u_b;B'_{2r_0}),
\end{align*}
where
\begin{equation*}
N(u_b;B'_{r})=r^2\norm{\nabla u_b}^2_{L^\infty(B'_r)}+r^4\norm{\nabla^2 u_b}_{L^\infty(B'_r)},
\end{equation*}
and the constant in the above inequality depends on $\lambda$, $\Lambda$, $f$ and $M$.
\end{thm}

\begin{lem}\label{l:largescales_bdry_D}
Assume \eqref{A1}, \eqref{B1} and \eqref{A2}.
There exist $\delta_0,\e_0,\eta_0>0$ and $\theta_0\in (0,1/2)$ (depending on $n$, $\lambda$, $\Lambda$ and $f$) such that any minimizer $u_\e$ of $F_\e(\cdot;B_1^+)$ with $\e\in (0,\e_0)$, $u_\e=u_b$ on $B'_1$  satisfies
\begin{align*}
&F_\e(u_\e;B_1^+) \leq \delta_0^2\quad\\
&\Longrightarrow \quad \theta_0^{2-n}F_\e(u_\e,B_{\theta_0}^+) \leq \frac 1{2} \max\left(F_\e(u_\e;B_1^+),\eta_0\norm{\nabla u_b}^2_{L^\infty(B'_1)} \right).
\end{align*}
\end{lem}

\begin{proof}
As for Lemma~\ref{l:largescales},
the proof is by contradiction, assuming the existence of  sequences $\e_\ell$, $u_\ell$, $W_\ell$, $a_\ell$, $u_b^\ell$ with $F_{\e_\ell}(u_\ell;B_1^+)\to 0$ and $\norm{\nabla u_b^\ell}_\infty \ll F_{\e_\ell}(u_\ell;B_1^+)$, but such that the energy decay fails. Then one needs three ingredients:
\begin{itemize}
\item the boundary modification Lemma~\ref{l:modifboundary} to construct a sequence of $\mathcal N$-valued minimizing maps $\overline w_\ell$,
\item the $H^1$-compactness of $v_\ell = \alpha_\ell^{-1}(\overline w_\ell -z_\ell)$, where $\alpha_\ell^2 = \int_{B_1^+}W(\overline w_\ell)$ and $z_\ell\in\mathcal N$ is appropriately chosen,
\item and the equivalent energy decay estimate for minimizers of $\int_{B_1^+} W(\nabla v)$ under the linear constraint $v\in T_z\mathcal N$ a.e., with constant boundary data on $B_1'$,
\end{itemize}
Lemma~\ref{l:modifboundary} can be applied here without modification, since $B_1^+$ is bilipschitz equivalent to $B_1$ and $u_\e=u_b$ is already $\mathcal N$-valued on $B'_1$. The compactness of $v_\ell$ follows as in Lemma~\ref{l:largescales} from the argument in \cite[Proposition~1]{luckhaus88}, where the extension lemma~\cite[Lemma~1]{luckhaus88} can also be applied without modification thanks to the bilipshitz homeomorphism between $B_1$ and $B_1^+$. The energy decay for $T_z\mathcal N$ minimizers comes from standard elliptic estimates. The rest of the proof is as in Lemma~\ref{l:largescales}.
\end{proof}

 The following lemma gives the decay estimate at finer scales; the difference between this and the corresponding estimate for the interior is a boundary term that, as it will be seen, behaves well under rescaling because it only involves derivatives of the boundary data $u_b .$
\begin{lem}\label{l:smallscales_bdry_D}
Assume \eqref{A1},\eqref{B1},\eqref{A2}, and \eqref{A3a} or \eqref{A3b}. There exists $\alpha>0$ (depending on $n$, $\lambda$, $\Lambda$, $f$, but also on $\e$), such that the following holds.
For all $\e>0$, any minimizer $u_\e$ of $F_{\e}(\cdot;B^+_2)$ with $u_\e=u_b$ on $B'_2$ and $F_\e(u_\e;B_2^+)\leq 1$ satisfies
\begin{equation*}
\frac{1}{r^\alpha}\frac{1}{r^{n-2}}\int_{B_r^+}\abs{\nabla u_\e}^2\lesssim E_{\e}(u_\e;B^+_2)+\norm{\nabla u_b}^2_{L^\infty(B'_2)}+\norm{\nabla^2u_b}^2_{L^\infty(B'_2)}\qquad\forall r\in (0,1),
\end{equation*}
where the inequality is up to a constant depending on $n$, $\lambda$, $\Lambda$, $f$, but also on $\e$.
\end{lem}
\begin{proof}
The proof can be carried out as the proof of Lemma~\ref{l:smallscales}, replacing the interior elliptic estimates \eqref{eq:ellipt_estim} with boundary elliptic estimates. More precisely, still denoting by $\mathcal L$ the elliptic operator such that $\int W(\nabla \varphi)a(x)dx=\int \mathcal L\varphi\cdot\varphi\, a(x)dx$ for all test functions $\varphi$, solutions of the Dirichlet problem
\begin{equation*}
\mathcal L u =f\text{ in }B_1,\quad u=g\text{ on }B'_1,
\end{equation*}
satisfy, for all $0<R_1<R_2\leq 1$,
\begin{equation*}
\norm{\nabla^2u}_{L^p(B_{R_1})}\lesssim \norm{f}_{L^p(B_{R_2})}+\frac{1}{(R_2-R_1)^2}\left(\norm{\nabla u}_{L^p(B_{R_2})}+\norm{\nabla g}_{L^\infty(B'_{R_2})}+\norm{\nabla^2g}_{L^\infty(B'_{R_2})}\right),
\end{equation*}
where the inequality is up to a constant depending on $n$, $\lambda$, $\Lambda$, $p$, $R_1$ and $R_2$ (this follows from the estimates in \cite[\S~10]{ADN2}). This is enough to reproduce the proof of Lemma~\ref{l:smallscales}, under either \eqref{A3a} or \eqref{A3b}.
\end{proof}

\begin{proof}[Proof of Theorem~\ref{t:bdry_D}]
For $x_0'\in B'_{r_0}$ we set $\bar u(\bar x)=u_\e(x_0'+r_0\bar x)$, $\bar u_b(\bar x')=u_b(x'_0+r_0\bar x')$ and $\bar \e =\e/\e_0$. Provided $\delta$ is small enough we can argue as in the proof of Theorem~\ref{t:int} and iterate Lemma~\ref{l:largescales_bdry_D} to obtain
\begin{equation*}
\bar r^{2-n}F_{\bar\e}(\bar u;B^+_{\bar r})\lesssim \bar r^{\alpha_0}\left(\int_{B_1^+}W(\nabla\bar u)a(x)dx +  \norm{\nabla \bar u_b}_{L^\infty(B'_1)}\right)\qquad\forall \bar r\in [\bar\e/\e_0,1).
\end{equation*}
Then we set $r_1=\bar\e/\e_0$ and $\hat u(\hat x)=\bar u(r_1\hat x)$, $\hat u_b(\hat x')=\bar u_b(r_1\hat x')$, $\hat \e=\bar e/r_1 =\e_0$ and apply Lemma~\ref{l:smallscales_bdry_D} to deduce for all $\hat r\in (0,1)$,
\begin{align*}
\frac{1}{\hat r^{\hat\alpha}}\frac{1}{\hat r^{n-2}}\int_{B_{\hat r}^+}\abs{\nabla \hat u}^2 & \lesssim F_{\hat\e}(\hat u;B_2^+)+\norm{\nabla\hat u}^2_{L^\infty(B'_2)}+\norm{\nabla^2\hat u}^2_{L^\infty(B'_2)}\\
&= r_1^{2-n}F_{\bar\e}(\bar u; B_{2r_1}^+)+r_1^2\norm{\nabla\bar u_b}_{L^\infty(B'_{2r_1})}+r_1^4\norm{\nabla^2\bar u_b}_{L^\infty(B'_{2r_1})}\\
&\lesssim r_1^{\alpha_0} \left(\int_{B_1^+}W(\nabla\bar u)a(x)dx +  N(\bar u_b;B'_1)\right).
\end{align*}
Rescaling and setting $\alpha=\min(\hat\alpha,\alpha_0)$ we infer
\begin{equation*}
\frac{1}{\bar r^{n-2}}\int_{B_{\bar r}^+}\abs{\nabla \bar u}^2 \lesssim r^{\alpha} \left(\int_{B_1^+}W(\nabla\bar u)a(x)dx +  N(\bar u_b;B'_1)\right).
\end{equation*}
Coming back to the original map $u$, the estimate above implies that
\begin{equation*}
\frac{1}{r^{n-2}}\int_{B_r(x_0')\cap B_{2r_0}^+}\abs{\nabla u}^2\lesssim \left(\frac r{r_0}\right)^\alpha
\left(\int_{B_{2r_0}^+}W(\nabla\bar u)a(x)dx + N(\bar u_b;B'_{2r_0})\right)\qquad\forall r\in (0,r_0),
\end{equation*}
provided $x_0'\in B_{r_0}'$. 

Next we consider $x_0\in B_{r_0}^+$ and write $x_0=(x_0',\rho)$ for some $\rho\in (0,r_0)$ and $x_0'=(x_0',0)\in B'_{r_0}$.
For all $r\in [\rho,r_0]$, we have $B_{r}(x_0)\cap B_{2r_0}^+\subset B_{2r}(x_0')$ and therefore by the above,
\begin{align*}
\frac{1}{r^{n-2}}\int_{B_r(x_0)\cap B_{2r_0}^+}\abs{\nabla u}^2 &
\lesssim \frac{1}{(2r)^{n-2}}\int_{B_{2r}(x_0')\cap B_{2r_0}^+}\abs{\nabla u}^2\\
&\lesssim r^{\alpha} \left(\int_{B_{2r_0}^+}W(\nabla\bar u)a(x)dx +   N(\bar u_b;B'_{2r_0})\right),
\end{align*}
whence in particular
\begin{equation*}
\frac{1}{(\rho/2)^{n-2}}\int_{B_{\rho/2}(x_0)\cap B_{2r_0}^+}\abs{\nabla u}^2\lesssim \delta^2.
\end{equation*}
Since $B_{\rho}(x_0)\subset B_{2r_0}^+$, provided $\delta$ is small enough we may therefore apply the interior estimates (Theorem~\ref{t:int}) in $B_{\rho}(x_0)$, and conclude that
\begin{equation*}
\frac{1}{r^{n-2}}\int_{B_r(x_0)\cap B_{2r_0}^+}\abs{\nabla u}^2 \lesssim r^{\alpha} \left(\int_{B_{2r_0}^+}W(\nabla\bar u)a(x)dx +   N(\bar u_b;B'_{2r_0})\right)
\end{equation*}
holds for all $r\in (0,r_0]$ and all $x_0\in B_{r_0}^+$. This implies the desired $C^{\alpha/2}$ H\"older estimate.
\end{proof}

\subsection{Weak anchoring}\label{ssWA}

We will denote by $F_\e^{wa}$ the energy
\begin{equation*}
F_\e^{wa}(u;B_2^+)=F_\e(u;B_2^+)+\int_{B'_2}g(x',u)\, dx',
\end{equation*}
where $g\colon B'_2\times \R^k\to [0,\infty)$ is a smooth anchoring energy density. Here we will always work under the assumption of a uniform $L^\infty$ bound \eqref{A3a}, and may therefore assume
\begin{equation}\tag{G}\label{G}
\norm{g}_{C^2(B'_2\times B_M)}\leq G,
\end{equation}
for some $G>0$. We will consider minimizers $u_\e$ of $F_\e^{wa}$ with respect to their own boundary conditions on $(\partial B_2)^+$. We will denote by $F_\star^{wa}$ the corresponding limiting energy, i.e. the same energy restricted to $\mathcal N$-valued maps.

\begin{thm}\label{t:bdry_WA_a}
Assume that $f$ satisfies \eqref{A2}, that $W$ satisfies \eqref{A1} and $a$ satisfies \eqref{B1}, and moreover that \eqref{A3a} and \eqref{G} hold.
There exist $\delta,\e_0>0$ and $\alpha\in (0,1)$ (depending on $n$, $\lambda$, $\Lambda$, $f$,  $M$ and $G$) such that for any $r_0\in (0,1)$, $\e\in (0,r_0\e_0)$, and any
 $u_\e$ minimizing $F_\e^{wa}(\cdot;B^+_{2r_0})$ with respect to its own boundary conditions on $(\partial B_{2r_0})^+$,
\begin{align*}
&(2r_0)^{2-n}F_\e(u_\e;B^+_{2r_0})+r_0\norm{g}_{L^\infty(B'_{2r_0})} \leq \delta^2\quad\\
&\Longrightarrow\quad r_0^{2\alpha} \abs{u_{\e_\ell}}^2_{C^\alpha(B^+_{r_0})}\lesssim  (2r_0)^{2-n}F_\e^{wa}(u_\e;B^+_{2r_0}),
\end{align*}
where
the constant in the above inequality depends on $n$, $\lambda$, $\Lambda$, $f$, $M$ and $G$.
\end{thm}

\begin{proof}
As above, this small energy estimate is a consequence of an energy improvement result ensuring regularity at large scales (Lemma~\ref{l:largescales_bdry_WA_a} below), and one establishing a corresponding property at small scales (Lemma~\ref{l:smallscales_bdry_WA_a} below). The proof is a straightforward adaptation of  Theorem~\ref{t:bdry_D}, given the following crucial scaling property: if $u$ minimizes $F_\e^{wa}$ in $B_r$, then $\tilde u(\tilde x):=u(r\tilde x)$ minimizes $\widetilde F_{\tilde \e}^{wa}$ in $B_1$ where $\tilde\e=\e/r$ and $\widetilde F_{\tilde \e}^{wa}$ corresponds to  $\widetilde W(\tilde x,\xi)=W(r\tilde x,\xi)$, $\tilde a(\tilde x)=a(r\tilde x)$ and, most importantly, $\tilde g(\tilde x,u)=r g(r\tilde x, u)$. Hence as we rescale $\norm{g}_{L^\infty}$ keeps getting smaller and this is what makes the iteration work.
\end{proof}

\begin{rem}
In the special case of Ginzburg-Landau functionals where $W=\abs{\nabla u}^2$, $f(u)=(1-\abs{u}^2)^2$ and $g(x',u)=\abs{u-u_b(x')}^2$, uniform convergence up to the boundary is proved for critical points in the recent work \cite{BPW18}. For more general anchoring energies however, and even in the isotropic case $W=\abs{\nabla u}^2$, this was not known before the present work.
 To prove this result, we need to assume \eqref{A3a}, a property that holds in the isotropic case under rather mild assumptions, e.g. $u\cdot f(u)\geq 0$ and $u\cdot\nabla_u g(u)\geq 0$ for $\abs{u}\geq M$ (in fact in that case \eqref{A3a} holds for all critical points provided the Euler-Lagrange equations are satisfied, see e.g. \cite{lamy14}). Here, the main reason for not being able to drop \eqref{A3a} is that otherwise we are unable to construct an extension $\varphi$ as in Lemma~\ref{l:modifboundary}, that  satisfies in addition a bound on $\int g(x',\varphi(x'))$. This impedes obtaining an equivalent of Lemma~\ref{l:smallscales} or \ref{l:smallscales_bdry_D}, which is essential to deal with \enquote{large} scales $r\geq \e$. On the other hand, regarding small scales (i.e. an equivalent of Lemma~\ref{l:largescales} or \ref{l:smallscales_bdry_D}), requiring \eqref{A3b}  together with some physically motivated restrictions on $g(x',u)$, is enough to ensure the desired estimate, even in the absence of \eqref{A3a}.
\end{rem}

\begin{lem}\label{l:largescales_bdry_WA_a}
Assume \eqref{A1}, \eqref{B1}, \eqref{A2} and \eqref{A3a}.
There exist $\delta_0,\e_0,\eta_0>0$ and $\theta_0\in (0,1/2)$ (depending on $n$, $\lambda$, $\Lambda$ and $f$) such that any minimizer $u_\e$ of $F_\e^{wa}(\cdot;B_1^+)$ with $\e\in (0,\e_0)$,  satisfies
\begin{align*}
&F_\e^{wa}(u_\e;B_1^+) \leq \delta_0^2\quad\\
&\Longrightarrow \quad \theta_0^{2-n}F_\e(u_\e,B_{\theta_0}^+) \leq \frac 1{2} \max\left(F_\e(u_\e;B_1^+),\eta_0\norm{g}_{L^\infty(B_1'\times B_M)} \right).
\end{align*}
\end{lem}
\begin{proof}
As for Lemmas~\ref{l:largescales} and \ref{l:largescales_bdry_D}, the proof is by contradiction, assuming the existence of sequences $\e_\ell$, $u_\ell$, $W_\ell$, $a_\ell$, $g_\ell$ with $F_{\e_\ell}(u_\ell;B_1^+)\to 0$ and $\norm{g_\ell}_\infty \ll F_{\e_\ell}(u_\ell;B_1^+)$, but such that the energy decay fails. Extending $u_\ell$ by symmetry to $B_1$, one may apply the boundary modification Lemma~\ref{l:modifboundary}, the weak anchoring energy of $\varphi$ simply being controlled by $\int_{B'_1\setminus B'_{1-\lambda}} g(\varphi)\lesssim \lambda\norm{g_\ell}_{L^\infty}$. Thus one obtains a sequence of minimizing $\mathcal N$-valued maps $\overline w_\ell$ with weak anchoring. The strong $H^1$ compactness of $v_\ell=\alpha_\ell^{-1}(w_\ell-z_\ell)$ is then obtained as in \cite[Proposition~1]{luckhaus88}, adapting the extension lemma \cite[Lemma~1]{luckhaus88} by first extending the maps by symmetry, as also explained in \cite{bdryregweakanchor}. The limit $v$ is then a minimizer of $\int_{B_1^+} W(v)$ with free boundary conditions on $B_1'$, and under the linear constraint $v\in T_z\mathcal N$ a.e., and enjoys good energy decay thanks to classical elliptic estimates.
\end{proof}

\begin{lem}\label{l:smallscales_bdry_WA_a}
Assume \eqref{A1},\eqref{B1},\eqref{A2}, \eqref{A3b} and \eqref{G}. 
For all $\e>0$, any minimizer $u_\e$ of $F_{\e}^{wa}(\cdot;B^+_2)$ satisfies
\begin{equation*}
\norm{\nabla u_\e}_{L^\infty(B_1^+)}^2\lesssim F_{\e}(u_\e;B^+_2)+\norm{g}^2_{C^2(B'_2\times B_M)}\qquad\forall r\in (0,1),
\end{equation*}
where the inequality is up to a constant depending on $n$, $\lambda$, $\Lambda$, $f$, $M$, $G$ but also on $\e$.
\end{lem}
\begin{proof}
As in Lemmas~\ref{l:smallscales} and \ref{l:smallscales_bdry_D}, the proof relies on elliptic estimates for the equation satisfied by $u=u_\e$, namely
\begin{equation*}
\left\lbrace \begin{aligned}
\mathcal L u &=\frac{1}{\e^2}\nabla f(u)&\text{in }B_2^+,\\
\mathcal B u &=\nabla_u g(x',u) &\text{on }B'_2,
\end{aligned}
\right.
\end{equation*}
where
\begin{equation*}
(\mathcal L u )^\beta=a^{-1}\partial_j(a\cdot a_{ij}^{\alpha\beta}\partial_i u^{\alpha}),\quad \mathcal (\mathcal B u)^\beta =-a\cdot a_{in}^{\alpha\beta}\partial_i u^\alpha.
\end{equation*}
We appeal to classical $L^p$ estimates for elliptic systems \cite[\S~10]{ADN2}, which ensure
\begin{equation*}
\norm{\nabla^2v}_{L^p(B_1^+)}\lesssim \norm{\nabla v}_{L^p(B_2^+)}+\norm{\mathcal L v}_{L^p(B_2^+)}+\norm{\nabla\Phi}_{ L^p(B_2^+)}\quad\forall\Phi\text{ such that }\mathcal B v = \tr \Phi,
\end{equation*}
where the inequality is up to a constant depending on $\lambda$, $\Lambda$, $n$ and $p\in (1,\infty)$. To apply this to our map $u$ we consider an extension $G$ of $g$ given by $G(x,u)=\chi(x_n)g(x',u)$ where $x=(x',x_n)$ and $\chi$ is a fixed smooth function with $\chi(0)=1$ and $\chi\equiv 0$ on $(1,\infty)$. That way we can use $\Phi=\nabla_u G(x,u)$ in the above and estimate
\begin{align*}
\norm{\nabla [\nabla_u G(x,u)]}_{L^p(B_2^+)}\lesssim \norm{g}_{C^2(B'_2\times B_M)} + \norm{\nabla u}_{L^p(B_2^+)}.
\end{align*}
We deduce that $u$ satisfies
\begin{equation*}
\norm{\nabla^2 u}_{L^p(B_1^+)}\lesssim \norm{\nabla u}_{L^p(B_2^+)} + \norm{\nabla f(u)}_{L^p(B_2^+)} 
+ \norm{g}_{C^2(B'_2\times B_M)},
\end{equation*}
and this estimate can be bootstrapped exactly as in the proof of Lemma~\ref{l:smallscales}.
\end{proof}

\appendix

\section{Proof of the boundary modification lemma}\label{a:modifboundary}

In this section we prove Lemma~\ref{l:modifboundary}. For the reader's convenience we recall here its statement:

\medskip
\noindent \textbf{Lemma~\ref{l:modifboundary}}. {\it
There exists $\delta_1=\delta_1(\mathcal N,f)>0$ such that for all $0< \e \leq \lambda<1$ and any $u\in H^1(\partial B_1;\R^k)$ with $E_\e(u;\partial B_1)\leq\delta_1^2\lambda^{n-3}$, there exist
\begin{align*}
& w\in H^1(\partial B_1;\mathcal N),\quad\varphi\in H^1(B_1\setminus B_{1-\lambda};\R^k),\\
& \text{with }\varphi = u\text{ on }\partial B_1,\quad \varphi= w((1-\lambda)\cdot)\text{ on }\partial B_{1-\lambda},
\end{align*}
satisfying the bounds
\begin{equation*}
E_\e(\varphi;B_1\setminus B_{1-\lambda})\lesssim \lambda E_\e(u;\partial B_1)\quad\text{and}\quad \int_{\partial B_1} \abs{\nabla w}^2 \lesssim E_\e (u;\partial B_1).
\end{equation*}
}
\medskip

\begin{proof}[Proof of Lemma~\ref{l:modifboundary}]
The strategy  is very similar to Luckhaus' extension lemma \cite[Lemma~1]{luckhaus88}. For the reader's convenience we sketch the full argument, and will go into details only at points where we need to depart from \cite{luckhaus88}.
We assume $\lambda=2^{-\nu}$ for some $\nu\in\mathbb N$ and, using the bilipschitz equivalence of $B_1$ with the open unit cube,  obtain a partition of $\partial B_1$ as
\begin{equation*}
\partial B_1=\bigsqcup_{j=0}^{n-1} Q_j,\quad Q_j=\bigsqcup_{i=1}^{k_j}e_{i}^j,
\end{equation*}
where each $j$-cell $e_{i}^j$ is bilipschitz equivalent to $B_\lambda^j$, the $j$-dimensional open ball  of radius $\lambda$. 
This decomposition of $\partial B_1$ induces a partition of $B_1\setminus B_{1-\lambda}$ as
\begin{equation*}
B_1\setminus B_{1-\lambda} =\bigsqcup_{j=0}^{n-1} \widehat Q_j,\quad \widehat Q_j =\bigsqcup_{i=1}^{k_j}\hat e_{i}^j,\quad
\hat e_i^j =
\left\lbrace x\in B_1\setminus B_{1-\lambda}\colon \frac{x}{\abs{x}}\in e_i^j\right\rbrace.
\end{equation*}
Moreover by Fubini's theorem we may assume that
\begin{align*}
\int_{Q_j} \abs{\nabla u}^2 \, d\mathcal H^j 
&\lesssim \lambda^{j+1-n}\int_{\partial B_1}\abs{\nabla u}^2\, d\mathcal H^{n-1},\\
\int_{Q_j} f(u)\, d\mathcal H^j
&\lesssim \lambda^{j+1-n}\int_{\partial B_1} f(u)\, d\mathcal H^{n-1}.
\end{align*} 
On the boundary of each two-dimensional cell $e_i^2$ (which is composed of $4$ one-dimensional cells) there holds
\begin{align*}
(\osc_{\partial e_i^2} u)^2 &\leq \lambda\int_{\partial e_i^2}\abs{\nabla u}^2\lesssim \lambda^{3-n}E_\e(u;\partial B_1)\lesssim \delta_1^2,\\
\fint_{\partial e_i^2}f(u)  & \lesssim \e^2 \lambda^{1-n}E_\e(u;\partial B_1)\lesssim \lambda^{3-n}E_\e(u;B_1)\lesssim \delta_1^2.
\end{align*}
If $\delta_1$ is small enough, this implies thanks to \eqref{A2} that
\begin{equation*}
\sup_{\partial e_i^2}\big\vert u - \fint_{\partial e_i^2}u\big\vert^2 + \dist^2(\fint_{\partial e_i^2}u,\mathcal N) \lesssim \lambda^{3-n}E_\e(u;\partial B_1)\lesssim\delta_1^2.
\end{equation*}
Therefore, the harmonic extension $\bar u$ of $u_{\lfloor \partial e_i^2}$ to $e_i^2$, i.e. $\Delta\bar u=0$ in $e_i^2$ and $\bar u=u$ on $\partial e_i^2$, satisfies
\begin{equation*}
\sup_{e_i^2}\dist^2(\bar u,\mathcal N) \lesssim \lambda^{3-n}E_\e(u;\partial B_1) \lesssim \delta_1^2.
\end{equation*}
In particular, provided $\delta_1$ is small enough, $w:=\pi_{\mathcal N}(\bar u)$ is well defined in $e_i^2$. Note for later use that $\bar u=u$ on $Q_1$, and thus
\begin{equation*}
w=\pi_{\mathcal N}(u)\text{ on }Q_1.
\end{equation*}
The harmonic extension $\bar u$ satisfies
\begin{equation*}
\int_{e_i^2}\abs{\nabla\bar u}^2\, d\mathcal H^{2} \lesssim \lambda \int_{\partial e_i^2}\abs{\nabla u}^2\, d\mathcal H^{1},
\end{equation*}
and therefore
\begin{equation*}
\int_{Q_2}\abs{\nabla w}^2 \, d\mathcal H^{2}\lesssim \int_{Q_2}\abs{\nabla\bar u}^2\, d\mathcal H^{2} \lesssim \lambda\int_{Q_1}\abs{\nabla u}^2 \, d\mathcal H^{1} \lesssim \lambda^{3-n}E_\e(u_\e;\partial B_1).
\end{equation*}
On the higher dimensional skeletons $Q_j$ ($j\geq 3$) we define the $\mathcal N$-valued map $w$ by induction, via 0-homogeneous extensions: identifying $e_i^{j}$ with $B_\lambda^{j}$ (through a bilipschitz homeomorphism that we omit to write here), one may set
\begin{equation*}
w(x)=w\left(\lambda\frac{x}{\abs{x}}\right)
\end{equation*}
thus defining $w$ on $e_i^{j}$ through its boundary values on $\partial  e_i^{j}$, which correspond to previously defined values of $w$ on $Q_{j-1}$. Since $j\geq 3$, such $0$-homogenous extension has finite energy, and we obtain by induction the estimates
\begin{align*}
\int_{Q_j}\abs{\nabla w}^2 \, d\mathcal H^{j}
&\lesssim \lambda^{j+1-n} E_\e(u_\e;\partial B_1).
\end{align*}
For $j=n-1$ this shows that
\begin{equation*}
\int_{\partial B_1}\abs{\nabla w}^2 \, d\mathcal H^{n-1} \lesssim E_\e(u_\e;\partial B_1).
\end{equation*}
It remains to define the map $\varphi$ on $B_1\setminus B_{1-\lambda}$. We do it on each skeleton $\widehat Q_j$ by induction, similarly to what is done in \cite[Lemma~1]{luckhaus88}. On each cell $\hat e_i^1$ we set
\begin{equation*}
\varphi(x)=u\left(\frac{x}{\abs{x}}\right) + \frac{1-\abs{x}}{\lambda}\left(w\left(\frac{x}{\abs{x}}\right) - u\left(\frac{x}{\abs{x}}\right) \right).
\end{equation*}
Since $\dist(u,\mathcal N)\lesssim\delta_1$ and  $w=\pi_{\mathcal N}(u)$ on $Q_1$, we have
\begin{align*}
f(\varphi)&\lesssim \dist^2(\varphi,\mathcal N)= \abs{u-\pi_{\mathcal N}(u)}^2\\
&\lesssim f(u) \qquad\text{a.e. on }\hat e_i^1,
\end{align*}
and therefore
\begin{equation*}
\int_{\widehat Q_1}f(\varphi) \, d\mathcal H^{2} \lesssim \lambda\int_{Q_1} f(u)\,d\mathcal H^1 \lesssim \lambda^{3-n}\int_{\partial B_1}f(u)\,d\mathcal H^{n-1}.
\end{equation*}
We also have, recalling that $w=\pi_{\mathcal N}(u)$ on $Q_1$ and $\e\leq\lambda$,
\begin{align*}
\int_{\widehat Q_1}\abs{\nabla\varphi}^2 \, d\mathcal H^2 &\lesssim \lambda \left(
\int_{Q_1}\abs{\nabla u}^2\, d\mathcal H^1 + \int_{Q_1}\abs{\nabla w}^2 \,d\mathcal H^1 + \frac{1}{\lambda^2}\int_{Q_1}\abs{u-w}^2\, d\mathcal H^1
\right) \\
&\lesssim \lambda \int_{Q_1}\abs{\nabla u}^2\,d\mathcal H^1 + \frac 1\lambda \int_{Q_1}f(u)\,d\mathcal H^1\\
&\lesssim  \lambda E_\e(u;Q_1) \lesssim \lambda^{3-n}E_\e(u;\partial B_1).
\end{align*}
On the higher dimensional skeletons $\widehat Q_j$ ($j\geq 2$) we extend $\varphi$ by induction, via $0$-homogenous extensions: identifying $\hat e_i^j$ with $B_\lambda^{j+1}$ (through a bilipschitz homeomorphism that we omit to write here), one may set
\begin{equation*}
\varphi(x)=\varphi\left(\lambda \frac{x}{\abs{x}}\right),
\end{equation*}
thus defining $\varphi$ on $\hat e_i^j$ through its boundary values on $\partial \hat e_i^j$, which either correspond to values of $u$ and $w$ on $Q_{j}$ or to previously defined values of $\varphi$ on $\widehat Q_{j-1}$. Since $j\geq 2$, such $0$-homogenous extension has finite energy, and we obtain by induction the estimates
\begin{align*}
E_\e(\varphi;\widehat Q_j)
&\lesssim \lambda^{j+2-n} E_\e(u;\partial B_1).
\end{align*}
For $j=n-1$ this concludes the proof.
\end{proof}

\section{Local compactness}\label{a:compact}

In this section we prove the following local compactness property of sequences of minimizers with bounded energy.

\begin{prop}\label{p:compact}
Assume that $f$ satisfies \eqref{A2} and $W$ satisfies \eqref{A1}. For $\e>0$, let $u_\e$ minimize $E_\e(\cdot;B_1)$ with respect to its own boundary conditions, such that $\liminf_{\e\to 0} E_\e(u_\e;B_1)<\infty$. Then there is a subsequence $\e_\ell\to 0$ such that $u_{\e_\ell}$ converges strongly in $H^1_{loc}(B_1;\R^k)$ to a map $u_\star\in H^1_{loc}(B_1;\mathcal N)$ which minimizes $E_\star(\cdot;B_\rho)$ for any $\rho\in (0,1)$ (among $\mathcal N$-valued maps, and with respect to its own boundary conditions).
\end{prop}

The crucial ingredient is the following variant of Luckhaus' extension lemma (see also \cite{canevari17} for related results in the Landau-de Gennes setting).

\begin{lem}\label{l:extension}
There exists $\eta\in (0,1)$ such that for all $\lambda\in (0,1)$ and any $u\in H^1(\partial B_1;\R^k)$, $v_\star\in H^1(\partial B_1;\mathcal N)$ with
\begin{align*}
\int_{\partial B_1}\abs{\nabla u}^2\, d\mathcal H^{n-1} + \int_{\partial B_1}\abs{\nabla v_\star}^2 \,d\mathcal H^{n-1}  \leq 1
\qquad\text{and}\qquad\int_{\partial B_1}\abs{u-v_\star}^2\, d\mathcal H^{n-1} \leq \eta^2\lambda^{2n-4},
\end{align*}
there exists $\varphi\in H^1(B_1\setminus B_{1-\lambda})$ such that
\begin{align*}
\varphi & = \begin{cases} u & \text{ on }\partial B_1,\\
 v_\star\left(\frac{\cdot}{1-\lambda}\right) & \text{ on }\partial B_{1-\lambda},
 \end{cases}\\
\int_{B_1\setminus B_{1-\lambda}}\abs{\nabla\varphi}^2 \, dx &\lesssim\lambda \left(
\int_{\partial B_1}\abs{\nabla u}^2\, d\mathcal H^{n-1} + \int_{\partial B_1}\abs{\nabla v_\star}^2 \,d\mathcal H^{n-1} + \frac{1}{\lambda^2}\int_{\partial B_1}\abs{u-v_\star}^2\, d\mathcal H^{n-1}
\right),\\
\int_{B_1\setminus B_{1-\lambda}}f(\varphi) \, dx &\lesssim \lambda \int_{\partial B_1} f(u)\, d\mathcal H^{n-1}.
\end{align*}
\end{lem}

\begin{proof}[Proof of Lemma~\ref{l:extension}]
The proof is very similar to \cite[Lemma~1]{luckhaus88}. For the reader's convenience we sketch the full argument, and will go into details only at points where we need to depart from \cite{luckhaus88}.
We assume $\lambda=2^{-\nu}$ for some $\nu\in\mathbb N$ and, using the bilipschitz equivalence of $B_1$ with the open unit cube,  obtain a partition of $\partial B_1$ as
\begin{equation*}
\partial B_1=\bigsqcup_{j=0}^{n-1} Q_j,\quad Q_j=\bigsqcup_{i=1}^{k_j}e_{i}^j,
\end{equation*}
where each $j$-cell $e_{i}^j$ is bilipschitz equivalent to $B_\lambda^j$, the $j$-dimensional open ball  of radius $\lambda$. 
This decomposition of $\partial B_1$ induces a partition of $B_1\setminus B_{1-\lambda}$ as
\begin{equation*}
B_1\setminus B_{1-\lambda} =\bigsqcup_{j=0}^{n-1} \widehat Q_j,\quad \widehat Q_j =\bigsqcup_{i=1}^{k_j}\hat e_{i}^j,\quad
\hat e_i^j =
\left\lbrace x\in B_1\setminus B_{1-\lambda}\colon \frac{x}{\abs{x}}\in e_i^j\right\rbrace.
\end{equation*}
Moreover by Fubini's theorem we may assume that
\begin{align*}
\int_{Q_j} \abs{\nabla u}^2 \, d\mathcal H^j 
+ \int_{Q_j}\abs{\nabla v_\star}^2\, d\mathcal H^j 
&\lesssim \lambda^{j+1-n}\left( \int_{\partial B_1}\abs{\nabla u}^2\, d\mathcal H^{n-1} + \int_{B_1}\abs{\nabla v_\star}^2\, d\mathcal H^{n-1}\right),\\
\int_{Q_j}\abs{u-v_\star}^2 \, d\mathcal H^j 
&\lesssim \lambda^{j+1-n}\int_{\partial B_1}\abs{u-v_\star}^2\, d\mathcal H^{n-1},\\
\int_{Q_j} f(u)\, d\mathcal H^j
&\lesssim \lambda^{j+1-n}\int_{\partial B_1} f(u)\, d\mathcal H^{n-1}.
\end{align*} 
In \cite{luckhaus88}, the extension $\varphi$ is defined on the $2$-dimensional cells $\hat e_i^1$ by interpolating linearly between $u$ on $e_i^1$ and $v_\star$ on $(1-\lambda) e_i^1$. However in our case we would like to control $\int_{\hat Q_1}f(\varphi)$. A simple linear interpolation may not be sufficient: consider e.g. the situation where $u(x)$ would happen to be on $\mathcal N$, hence $f(u(x))=0$, but the segment between $v_\star(x)$ and $u(x)$ might contain points which are not on $\mathcal N$, and there $f(\varphi)$ would not be controlled by $f(u)$. A way around this is to first interpolate linearly between $u$ and its projection $\pi(u)\in\mathcal N$, and then geodesically between $\pi(u)$ and $v_\star$ on $\mathcal N$.

On each one dimensional cell $e_i^1$, 
we have 
\begin{align*}
\sup_{e_i^1}\abs{u-v_\star}^2 & \leq \int_{e_i^1}\abs{\nabla \abs{u-v_\star}} + \frac 1\lambda \int_{e_i^1}\abs{u-v_\star}^2\\
&\leq \left(\int_{e_i^1}(\abs{\nabla u}^2+\abs{\nabla v_\star}^2)\right)^{\frac 12}\left(\int_{e_i^1}\abs{u-v_\star}^2\right)^{\frac 12}+\frac 1\lambda \int_{e_i^1}\abs{u-v_\star}^2\\
&\lesssim \lambda^{2-n}\eta \lambda^{n-2} + \lambda^{1-n}\eta^2\lambda^{4n-2} \lesssim \eta.
\end{align*}
Provided $\eta$ is chosen small enough, this implies that on $e_i^1$ the projection $\pi_{\mathcal N}(u)$ is well defined, and satisfies
\begin{equation*}
\abs{u-\pi_{\mathcal N}(u)} + \abs{\pi_{\mathcal N}(u)-v_\star}\lesssim \abs{u-v_\star}\lesssim\eta^{\frac 12}\quad\text{on }e_i^1.
\end{equation*}
Then we define $\varphi$ on $\hat e_i^1$ by setting
\begin{equation*}
\varphi(x)=\left\lbrace
\begin{aligned}
&u\left(\frac{x}{\abs{x}}\right) + 2\frac{1-\abs{x}}{\lambda}\left(\pi_{\mathcal N}(u)\left(\frac{x}{\abs{x}}\right) - u\left(\frac{x}{\abs{x}}\right) \right) &\text{if } 1-\frac \lambda 2 \leq \abs{x} \leq 1,\\
& \gamma\left(2\frac{1-\abs{x}}{\lambda}-1,\pi_{\mathcal N}(u)\left(\frac{x}{\abs{x}}\right), v_\star \left(\frac{x}{\abs{x}}\right) \right)& \text{if } 1-\lambda\leq \abs{x}\leq 1-\frac\lambda 2,\end{aligned}
\right.
\end{equation*}
Where $\gamma(\cdot,z_1,z_2)\colon [0,1]\to\mathcal N$ denotes the constant speed geodesic from $z_1$ to $z_2$. The map $\gamma$ is Lipschitz on a neighborhood of $[0,1]\times\Delta$, where $\Delta=\lbrace (z,z)\rbrace\subset\mathcal N\times\mathcal N$, and its derivatives satisfy
\begin{equation*}
\abs{\partial_t\gamma(t,z_1,z_2)}\lesssim \abs{z_1-z_2},\quad \abs{\nabla_z\gamma}\lesssim 1.
\end{equation*}
Using this, we infer that
\begin{align*}
\int_{\widehat Q_1}\abs{\nabla\varphi}^2 \, d\mathcal H^2 &\lesssim \lambda \left(
\int_{Q_1}\abs{\nabla u}^2\, d\mathcal H^1 + \int_{Q_1}\abs{\nabla v_\star}^2 \,d\mathcal H^1 + \frac{1}{\lambda^2}\int_{Q_1}\abs{u-v_\star}^2\, d\mathcal H^1
\right) \\
&\lesssim \lambda\cdot\lambda^{2-n}
\left(
\int_{\partial B_1}\abs{\nabla u}^2\, d\mathcal H^{n-1} + \int_{\partial B_1}\abs{\nabla v_\star}^2 \,d\mathcal H^{n-1} + \frac{1}{\lambda^2}\int_{\partial B_1}\abs{u-v_\star}^2\, d\mathcal H^{n-1}
\right).
\end{align*}
Moreover for $1-\lambda\leq\abs{x}\leq 1-\frac \lambda 2$ it holds $f(\varphi(x))=0$, and for $1-\frac\lambda 2 \leq\abs{x}\leq 1$ it holds 
\begin{align*}
f(\varphi(x))&\lesssim\dist^2(\varphi(x),\mathcal N)\lesssim \abs{u-\pi_{\mathcal N}(u)}^2\left(\frac{x}{\abs{x}}\right)\\
&\lesssim \dist^2(u,\mathcal N)\left(\frac{x}{\abs{x}}\right)\lesssim f(u)\left(\frac{x}{\abs{x}}\right),
\end{align*}
and this implies
\begin{equation*}
\int_{\widehat Q_1} f(\varphi) \, d\mathcal H^2\lesssim \lambda\int_{Q_1} f(u) \, d\mathcal H^1 \lesssim \lambda \cdot\lambda^{2-n}\int_{\partial B_1} f(u)\, d\mathcal H^{n-1}.
\end{equation*}

On the higher dimensional skeletons $\widehat Q_j$ ($j\geq 2$) we extend $\varphi$ by induction, via $0$-homogenous extensions: identifying $\hat e_i^j$ with $B_\lambda^{j+1}$ (through a bilipschitz homeomorphism that we omit to write here), one may set
\begin{equation*}
\varphi(x)=\varphi\left(\lambda \frac{x}{\abs{x}}\right),
\end{equation*}
thus defining $\varphi$ on $\hat e_i^j$ through its boundary values on $\partial \hat e_i^j$, which either correspond to values of $u$ and $v_\star$ on $Q_{j}$ or to previously defined values of $\varphi$ on $\widehat Q_{j-1}$. Since $j\geq 2$, such $0$-homogenous extension has finite energy, and we obtain by induction the estimates
\begin{align*}
\int_{\widehat Q_j}\abs{\nabla\varphi}^2 \, d\mathcal H^{j+1}
&\lesssim \lambda\cdot\lambda^{j+1-n}
\left(
\int_{\partial B_1}\abs{\nabla u}^2\, d\mathcal H^{n-1} + \int_{\partial B_1}\abs{\nabla v_\star}^2 \,d\mathcal H^{n-1} + \frac{1}{\lambda^2}\int_{\partial B_1}\abs{u-v_\star}^2\, d\mathcal H^{n-1}
\right),\\
\int_{\widehat Q_j} f(\varphi) \, d\mathcal H^{j+1} &\lesssim \lambda \cdot\lambda^{j+1-n}\int_{\partial B_1} f(u)\, d\mathcal H^{n-1}.
\end{align*}
For $j=n-1$ this concludes the proof.
\end{proof}

Now we turn to the proof of the compactness result.

\begin{proof}[Proof of Proposition~\ref{p:compact}]
Along a subsequence $\e_\ell\to 0$ we have $E_{\e_\ell}(u_\ell;B_1)\lesssim 1$, where we denote $u_\ell=u_{\e_\ell}$.
Therefore, up to taking a subsequence there is $u_\star\in H^1(B_1;\mathcal N)$ such that $u_{\e_\ell}\rightharpoonup u_\star$ weakly in $H^1$ and strongly in $L^2$. Let $\rho\in (0,1)$ be fixed. By Fubini's theorem we may find $r\in [\rho,1]$ such that $E_{\e_\ell}(u_\ell;\partial B_r)\leq 1$ and we infer that $u_\ell$ converges strongly towards $u_\star$ in $L^2(\partial B_r)$, and moreover $E_\star(u_\star;\partial B_r)\leq 1$.

Let $v_\star$ minimize $E_\star(\cdot;B_r)$ with $v_\star=u_\star$ on $\partial B_r$. Using Lemma~\ref{l:extension} we construct $v_\ell\in H^1(B_r;\R^k)$ such that $v_\ell=u_\ell$ on $\partial B_r$, and $E_{\e_\ell}(v_\ell;B_r)\to E_\star(v_\star;B_r)$.
Explicitly, since $\mu_\ell:=\int_{\partial B_r}\abs{u_\ell - v_\star}^2\to 0$, we may for large enough $\ell$ apply Lemma~\ref{l:extension} to find $\lambda_\ell\to 0$ and $\varphi_\ell\in H^1(B_r\setminus B_{(1-\lambda_\ell)r})$ satisfying $\varphi_\ell=u_\ell$ on $\partial B_r$, $\varphi_\ell=v_\star(\cdot/(1-\lambda_\ell))$ on $\partial B_{(1-\lambda_\ell)r}$ and $E_{\e_\ell}(\varphi_\ell;B_r\setminus B_{(1-\lambda_\ell)r})\lesssim \lambda_\ell \to 0$. Then we set $v_\ell=\varphi_\ell$ in $B_r\setminus B_{(1-\lambda_\ell)r}$ and $v_\ell = v_\star(\cdot/(1-\lambda_\ell))$ in $B_{(1-\lambda_\ell)r}$ and obtain indeed $E_{\e_\ell}(v_\ell;B_r)\to E_\star(v_\star;B_r)$.

By the minimizing property of $u_\ell$ we infer that $\limsup E_{\e_\ell}(u_\ell;B_r)\leq \limsup E_{\e_\ell}(v_\ell;B_r)\leq E_\star (v_\star;B_r)$, and since by weak lower semicontinuity $\liminf E_{\e_\ell}(u_\ell;B_r)\geq E_\star(u_\star; B_r)$  we conclude that $u_\star$ minimizes $E_\star(\cdot;B_r)$ and $u_\ell$ converges in fact strongly to $u_\star$ in $H^1(B_r)$.
\end{proof}

\bibliographystyle{acm}
\bibliography{convQ}

\begin{thebibliography}{10}

\bibitem{ADN2}
{\sc Agmon, S., Douglis, A., and Nirenberg, L.}
\newblock Estimates near the boundary for solutions of elliptic partial
  differential equations satisfying general boundary conditions. {II}.
\newblock {\em Comm. Pure Appl. Math. 17\/} (1964), 35--92.

\bibitem{alamabronsardlamy16phys}
{\sc Alama, S., Bronsard, L., and Lamy, X.}
\newblock Analytical description of the {S}aturn-ring defect in nematic
  colloids.
\newblock {\em Phys. Rev. E 93\/} (2016), 012705.

\bibitem{alamabronsardlamy16saturn}
{\sc Alama, S., Bronsard, L., and Lamy, X.}
\newblock Minimizers of the {L}andau--de {G}ennes energy around a spherical
  colloid particle.
\newblock {\em Arch. Ration. Mech. Anal. 222}, 1 (2016), 427--450.

\bibitem{alamabronsardlamy17}
{\sc Alama, S., Bronsard, L., and Lamy, X.}
\newblock Spherical particle in nematic liquid crystal under an external field:
  the {S}aturn ring regime.
\newblock {\em J. Nonlinear Sci. 28}, 4 (2018), 1443--1465.

\bibitem{ballzarnescu08}
{\sc Ball, J., and Zarnescu, A.}
\newblock Orientable and non-orientable line field models for uniaxial nematic
  liquid crystals.
\newblock {\em Mol. Cryst. Liq. Cryst. 495}, 1 (2008), 221/[573]--233/[585].

\bibitem{ballbedford15}
{\sc Ball, J.~M., and Bedford, S.~J.}
\newblock Discontinuous order parameters in liquid crystal theories.
\newblock {\em Mol. Cryst. Liq. Cryst. 612}, 1 (2015), 1--23.

\bibitem{ballmajumdar10}
{\sc Ball, J.~M., and Majumdar, A.}
\newblock Nematic liquid crystals: From maier-saupe to a continuum theory.
\newblock {\em Mol. Cryst. Liq. Cryst. 525}, 1 (2010), 1--11.

\bibitem{ballzarnescu11}
{\sc Ball, J.~M., and Zarnescu, A.}
\newblock Orientability and energy minimization in liquid crystal models.
\newblock {\em Arch. Ration. Mech. Anal. 202}, 2 (2011), 493--535.

\bibitem{baumanparkphillips12}
{\sc Bauman, P., Park, J., and Phillips, D.}
\newblock Analysis of nematic liquid crystals with disclination lines.
\newblock {\em Arch. Ration. Mech. Anal. 205}, 3 (2012), 795--826.

\bibitem{BPW18}
{\sc Bauman, P., Phillips, D., and Wang, C.}
\newblock Higher dimensional {G}inzburg-{L}andau equations under weak anchoring
  boundary conditions.
\newblock {\em J. Funct. Anal.\/} (2018).

\bibitem{bedford16}
{\sc Bedford, S.}
\newblock Function spaces for liquid crystals.
\newblock {\em Arch. Ration. Mech. Anal. 219}, 2 (2016), 937--984.

\bibitem{BBH1}
{\sc Bethuel, F., Brezis, H., and H{\'e}lein, F.}
\newblock {\em Ginzburg-{L}andau vortices}.
\newblock Progress in Nonlinear Differential Equations and their Applications,
  13. Birkh\"auser, 1994.

\bibitem{canevari2d}
{\sc Canevari, G.}
\newblock Biaxiality in the asymptotic analysis of a 2{D} {L}andau--de {G}ennes
  model for liquid crystals.
\newblock {\em ESAIM Control Optim. Calc. Var. 21}, 1 (2015), 101--137.

\bibitem{canevari17}
{\sc Canevari, G.}
\newblock Line defects in the small elastic constant limit of a
  three-dimensional {L}andau--de {G}ennes model.
\newblock {\em Arch. Ration. Mech. Anal. 223}, 2 (2017), 591--676.

\bibitem{chenlin93}
{\sc Chen, Y., and Lin, F.}
\newblock Evolution of harmonic maps with {D}irichlet boundary conditions.
\newblock {\em Comm. Anal. Geom 1}, 3-4 (1993), 327--346.

\bibitem{biaxialescape}
{\sc Contreras, A., and Lamy, X.}
\newblock Biaxial escape in nematics at low temperature.
\newblock {\em J. Funct. Anal. 272}, 10 (2017), 3987--3997.

\bibitem{bdryregweakanchor}
{\sc Contreras, A., Lamy, X., and Rodiac, R.}
\newblock Boundary regularity of weakly anchored harmonic maps.
\newblock {\em C. R. Math. Acad. Sci. Paris 353}, 12 (2015), 1093--1097.

\bibitem{singperturb}
{\sc Contreras, A., Lamy, X., and Rodiac, R.}
\newblock On the convergence of minimizers of singular perturbation
  functionals.
\newblock {\em Indiana Univ. Math. J. 67}, 4 (2018), 1665--1682.

\bibitem{difrattaetal16}
{\sc Di~Fratta, G., Robbins, J.~M., Slastikov, V., and Zarnescu, A.}
\newblock Half-{I}nteger {P}oint {D}efects in the {Q}-{T}ensor {T}heory of
  {N}ematic {L}iquid {C}rystals.
\newblock {\em J. Nonlinear Sci. 26}, 1 (2016), 121--140.

\bibitem{golovatymontero14}
{\sc Golovaty, D., and Montero, J.~A.}
\newblock On minimizers of a {L}andau--de {G}ennes energy functional on planar
  domains.
\newblock {\em Arch. Ration. Mech. Anal. 213}, 2 (2014), 447--490.

\bibitem{hkl86}
{\sc Hardt, R., Kinderlehrer, D., and Lin, F.-H.}
\newblock Existence and partial regularity of static liquid crystal
  configurations.
\newblock {\em Comm. Math. Phys. 105}, 4 (1986), 547--570.

\bibitem{hkl88}
{\sc Hardt, R., Kinderlehrer, D., and Lin, F.-H.}
\newblock Stable defects of minimizers of constrained variational principles.
\newblock {\em Ann. Inst. H. Poincar\'e Anal. Non Lin\'eaire 5}, 4 (1988),
  297--322.

\bibitem{henaomajumdarpisante17}
{\sc Henao, D., Majumdar, A., and Pisante, A.}
\newblock Uniaxial versus biaxial character of nematic equilibria in three
  dimensions.
\newblock {\em Calc. Var. Partial Differential Equations 56}, 2 (2017), Art.
  55, 22.

\bibitem{hong04}
{\sc Hong, M.-C.}
\newblock Partial regularity of weak solutions of the liquid crystal
  equilibrium system.
\newblock {\em Indiana Univ. Math. J. 53}, 5 (2004), 1401--1414.

\bibitem{ignatlamy17}
{\sc Ignat, R., and Lamy, X.}
\newblock Lifting of {$\mathbb{RP}^{d-1}$}-valued maps in {$BV$} and
  applications to uniaxial {$Q$}-tensors. {W}ith an appendix on an intrinsic
  {$BV$}-energy for manifold-valued maps.
\newblock {\em Calc. Var. Partial Differential Equations 58}, 2 (2019), Art.
  68, 26.

\bibitem{INSZuniqhedgehog}
{\sc Ignat, R., Nguyen, L., Slastikov, V., and Zarnescu, A.}
\newblock Uniqueness results for an {ODE} related to a generalized
  {G}inzburg-{L}andau model for liquid crystals.
\newblock {\em SIAM J. Math. Anal. 46}, 5 (2014), 3390--3425.

\bibitem{INSZstabhedgehog}
{\sc Ignat, R., Nguyen, L., Slastikov, V., and Zarnescu, A.}
\newblock Stability of the melting hedgehog in the {L}andau--de {G}ennes theory
  of nematic liquid crystals.
\newblock {\em Arch. Ration. Mech. Anal. 215}, 2 (2015), 633--673.

\bibitem{INSZinstab2d}
{\sc Ignat, R., Nguyen, L., Slastikov, V., and Zarnescu, A.}
\newblock Instability of point defects in a two-dimensional nematic liquid
  crystal model.
\newblock {\em Ann. Inst. H. Poincar\'e Anal. Non Lin\'eaire 33}, 4 (2016),
  1131--1152.

\bibitem{INSZstab2d}
{\sc Ignat, R., Nguyen, L., Slastikov, V., and Zarnescu, A.}
\newblock Stability of point defects of degree {$\pm\frac{1}{2}$} in a
  two-dimensional nematic liquid crystal model.
\newblock {\em Calc. Var. Partial Differential Equations 55}, 5 (2016), Paper
  No. 119, 33.

\bibitem{kitavtsev16}
{\sc Kitavtsev, G., Robbins, J.~M., Slastikov, V., and Zarnescu, A.}
\newblock Liquid crystal defects in the {L}andau--de {G}ennes theory in two
  dimensions---beyond the one-constant approximation.
\newblock {\em Math. Models Methods Appl. Sci. 26}, 14 (2016), 2769--2808.

\bibitem{kresinmazya}
{\sc Kresin, G., and Maz'ya, V.}
\newblock {\em Maximum principles and sharp constants for solutions of elliptic
  and parabolic systems}, vol.~183 of {\em Mathematical Surveys and
  Monographs}.
\newblock American Mathematical Society, Providence, RI, 2012.

\bibitem{lamy14}
{\sc Lamy, X.}
\newblock Bifurcation analysis in a frustrated nematic cell.
\newblock {\em J. Nonlinear Sci. 24}, 6 (2014), 1197--1230.

\bibitem{luckhaus88}
{\sc Luckhaus, S.}
\newblock Partial {H}\"older continuity for minima of certain energies among
  maps into a {R}iemannian manifold.
\newblock {\em Indiana Univ. Math. J. 37}, 2 (1988), 349--367.

\bibitem{majumdarzarnescu10}
{\sc Majumdar, A., and Zarnescu, A.}
\newblock Landau-{D}e {G}ennes theory of nematic liquid crystals: the
  {O}seen-{F}rank limit and beyond.
\newblock {\em Arch. Ration. Mech. Anal. 196}, 1 (2010), 227--280.

\bibitem{morrey}
{\sc Morrey, Jr., C.~B.}
\newblock {\em Multiple integrals in the calculus of variations}.
\newblock Classics in Mathematics. Springer-Verlag, Berlin, 2008.
\newblock Reprint of the 1966 edition [MR0202511].

\bibitem{mottramnewton}
{\sc Mottram, N., and Newton, C.}
\newblock Introduction to {Q}-tensor theory.
\newblock arXiv:1409.3542, 2014.

\bibitem{nguyenzarnescu13}
{\sc Nguyen, L., and Zarnescu, A.}
\newblock Refined approximation for minimizers of a {L}andau-de {G}ennes energy
  functional.
\newblock {\em Calc. Var. Partial Differential Equations 47}, 1-2 (2013),
  383--432.

\bibitem{schoen84}
{\sc Schoen, R.}
\newblock Analytic aspects of the harmonic map problem.
\newblock In {\em Seminar on nonlinear partial differential equations
  ({B}erkeley, {C}alif., 1983)}, vol.~2 of {\em Math. Sci. Res. Inst. Publ.}
  Springer, New York, 1984, pp.~321--358.

\end{thebibliography}

\end{document}